\documentclass[12pt]{amsart}
\usepackage{amsmath,amscd,amsbsy,amssymb,amsthm,amsfonts}
\usepackage{latexsym,url,bm}
\usepackage{cite,enumitem}
\usepackage{fullpage}
\usepackage{color}
\usepackage[colorlinks=true]{hyperref}
\usepackage{chngcntr}
\usepackage{apptools}

\newtheorem{theorem}{Theorem}[section]
\newtheorem{proposition}[theorem]{Proposition}
\newtheorem{lemma}[theorem]{Lemma}

\theoremstyle{remark}

\theoremstyle{definition}

\allowdisplaybreaks

\numberwithin{equation}{section}

\newcommand{\p}{\partial}

\newcommand{\R}{\mathbb{R}}
\newcommand{\N}{\mathbb{N}}
\newcommand{\Z}{\mathbb{Z}}
\newcommand{\C}{\mathbb{C}}

\renewcommand{\H}{{\mathcal H }}
\newcommand{\al}{\alpha}
\newcommand{\f}{\frac}

\newcommand{\la}{\lambda}

\newcommand{\bu}{\mathbf{u}}

\newcommand{\T}{{\mathbb T}}

\newcommand{\mA}{\mathcal{A}}
\newcommand{\mB}{\mathcal{B}}

\newcommand{\mD}{\mathcal{D}}

\newcommand{\mH}{\mathcal{H}}
\newcommand{\mK}{\mathcal{K}}
\newcommand{\mL}{\mathcal{L}}

\newcommand{\mS}{\mathcal{S}}

\newcommand{\mR}{\mathcal{R}}
\newcommand{\mM}{\mathcal{M}}

\newcommand{\mU}{\mathcal{U}}

\newcommand{\mZ}{\mathcal{Z}}

\renewcommand{\a}{\alpha}
\renewcommand{\b}{\beta}

\newcommand{\inv}{^{-1}}

\author{Lizhe Wan}
\address{Beijing International Center for Mathematical Research, Peking University}
\curraddr{}
\email{wanlizhe@pku.edu.cn}

\author{Jiaqi Yang}
\address{School of Mathematics and Statistics, Northwestern Polytechnical University}
\curraddr{}
\email{yjqmath@nwpu.edu.cn, yjqmath@163.com}

\begin{document}

\title{Exact controllability of two-dimensional hydroelastic waves}

\begin{abstract}
We prove the exact controllability of two-dimensional hydroelastic waves in the periodic setting.
We show that if the initial data and the final data are small, for exterior pressure  whose support is any non-empty open set $\omega$, the two-dimensional hydroelastic wave system is exactly controllable in arbitrary short time. 
\end{abstract}

\maketitle

\section{Introduction} \label{s:Intro}
We consider two-dimensional irrotational, inviscid and incompressible fluid, with the top surface being in contact with a frictionless thin elastic sheet.
The fluid occupies a time-dependent domain $\Omega_t \subseteq \mathbb{R}/2\pi \mathbb{Z} \times \mathbb{R}$ with a free periodic upper boundary $\Gamma_t$ and a fixed flat lower boundary $\Gamma_b = \{(x,y)|y= -h\}$ ($\Gamma_b$ is empty in the infinite depth case).
Denoting the fluid velocity by $(u(t,x,y), v(t,x,y))$, the pressure by $p(t,x,y)$, the equations inside the domain $\Omega_t$ are given by the following incompressible irrotational Euler equations:
\begin{equation*}
\left\{
             \begin{array}{lr}
            u_t +uu_x +vu_y = -p_x,&  \\
            v_t + uv_x +vv_y = -g-p_y,& \\
            u_x +v_y =0, \quad u_y -v_x = 0.&
             \end{array}
\right.
\end{equation*}
On the boundary $\Gamma_t$, we assume no surface tension. 
Instead, the dynamic boundary condition involves a nonlinear elastic restoring force together with an exterior pressure $P_{ext}$,
\begin{equation*}
    p = - \sigma\Delta_{\Gamma_t} \kappa -\f{\sigma}{2} \kappa^3+ P_{ext} \quad \text{on }\Gamma_t,
\end{equation*}
and the kinematic boundary condition is
\begin{equation*}
    \partial_t +(u,v)\cdot \nabla_{x,y} \text{ is tangent to }\Gamma_t.
\end{equation*}
Here $\kappa$ is the mean curvature of the surface $\Gamma_t$, $g> 0$ represents the gravitational acceleration, and $\sigma>0$ is the coefficient of flexural rigidity. 
In the finite depth case we also impose a boundary condition at the bottom:
\begin{equation*}
    v = 0 \quad \text{on } \Gamma_b.
\end{equation*}

This free boundary problem, commonly referred to as the two-dimensional \textbf{hydroelastic wave} system, was first introduced by Toland \cite{MR2413099} in the study of steady periodic waves.
The Cauchy problem of hydroelastic waves was studied by Ambrose and Siegel \cite{MR3656704}, Ambrose and Liu \cite{MR3608168}, as well as Wang and Yang \cite{MR4104949,MR4846707}.
More recently, a low-regularity well-posedness result was established in \cite{WanY}. 
In addition, authors \cite{WanY26} proved the well-posedness of the Muskat problem with an elastic interface. For results on solitary waves in hydroelastic waves, we refer to the works of Groves \textit{et al.} \cite{MR4733013,MR3566507}.

In this paper, we study the internal controllability of two-dimensional hydroelastic waves.
The problem can be written in the Zakharov-Craig-Sulem formulation.
Suppose the fluid domain $\Omega_t$ at time $t$ can be represented by
\begin{equation*}
    \Omega_t = \{ (x,y) \in \mathbb{T}\times\R; -h< y<\eta(t,x) \},
\end{equation*}
where $\eta(t,x): \R_t \times \R_x \rightarrow \R$ parametrizes the free surface $\Gamma_t$, and $h$ is the depth of the domain (which may also be $\infty$ in the case of infinite depth).
Without loss of generality, we assume that $\eta$ has zero mean.
We define the velocity potential $\phi$ such that $(u,v) = \nabla \phi$, where $\phi$ solves the equation
\begin{equation*}
\Delta_{x,y}\phi = 0, \quad \p_t \phi + \f12 |\nabla_{x,y} \phi|^2 + gy = p, \quad \p_y \phi |_{y = -h} = 0.
\end{equation*}
By the classical theory of elliptic equations, the problem reduces to the evolution of the velocity potential on the boundary $\Gamma_t$. 
Let \( \psi(t, x) = \phi(t, x, \eta(t, x)) \) denote the trace of the velocity potential on the free surface. 
The system can then be reformulated as:
\begin{equation}\label{HydroWave}
\begin{cases}
 \eta_t = G(\eta)\psi\,, \\
\psi_t  + g\eta +\frac{1}{2} \psi_x^2- \f12\f{(\eta_x\psi_x-G(\eta)\psi)^2}{1+\eta^2_x} +\sigma \mathbf{E}(\eta) = P_{ext}\,,
\end{cases}
\end{equation}
where $G(\eta)$ is the Dirichlet–Neumann operator defined by
\begin{equation*}
 G(\eta)\psi := \sqrt{1+\eta_x^2}\phi_n |_{y=\eta(t,x)} = (\phi_y - \eta_x \phi_x)|_{y=\eta(t,x)}\,,   
\end{equation*}
and the elastic term $\mathbf{E}(\eta)$ is given by
\begin{equation} \label{Elastic}
\begin{aligned}
    \mathbf{E}(\eta) :=&\sigma\left\{\frac{1}{\sqrt{1+\eta_x^2}} 
    \left[ \frac{1}{\sqrt{1+\eta_x^2}} 
    \left( \frac{\eta_{xx}}{(1+\eta_x^2)^{3/2}} \right)_x \right]_x
    + \frac{1}{2} \left( \frac{\eta_{xx}}{(1+\eta_x^2)^{3/2}} \right)^3\right\} \\
    =&\sigma\left\{\left(\f{\eta_{xx}}{(1+\eta^2_x)^{\f52}}\right)_{xx}+\f52\left(\f{\eta_x\eta_{xx}^2}{(1+\eta^2_x)^{\f72}}\right)_x \right\}.
\end{aligned}
\end{equation}

The hydroelastic wave system is a special kind of water wave problem with nonlinear elastic boundary condition.
The control theory for related water wave models is by now well developed, and we briefly recall the relevant results.
Reid \& Russell \cite{MR0774033} and Reid \cite{MR0846383, MR1348122} studied the controllability of linearized water wave equations at the origin.
Later, Alazard, Baldi and Han-Kwan showed in \cite{MR3776276} the controllability of the full gravity-capillary water wave system in two space dimensions using paradifferential calculus and an Ingham type inequality.
 Zhu in \cite{MR4072685} extended their result in higher dimensions using a different semiclassical approach.
Alazard  also proved the boundary observability of gravity water waves in two and three dimensions using the multiplier method \cite{MR3778651}.
In addition, Alazard obtained the stabilization results for both gravity water waves \cite{MR3801750} and gravity-capillary water waves \cite{MR3708590}.

Before stating the controllability result of the system \eqref{HydroWave}, we first briefly examine the control theory for the linearized hydroelastic waves around the null solution. 
The system is given by 
\begin{equation}
\left\{
             \begin{array}{lr}
             \partial_t \eta  = G(0) \psi, &  \\
             \p_t \psi + g\eta  + \sigma \p_x^4 \eta = P_{ext},&
             \end{array}
\right.\label{e:ZeroLinear}
\end{equation}
where $G(0) = |D|tanh(h|D|)$ ($G(0) = |D|$ in the infinite depth case). 
When $g = P_{ext} = 0$ and $h = +\infty$, the linear system \eqref{e:ZeroLinear} is well-posed in $H^{s+\f32}(\mathbb{T}) \times H^s(\mathbb{T})$ for any $s \in \R$.
We then define the product Sobolev space
\begin{equation*}
    \H^s : = H^{s+\f32}(\mathbb{T}) \times H^s(\mathbb{T}).
\end{equation*}
The inverse $G(0)^{-1}$ is well-defined for periodic functions with zero mean.
Introducing the linear operator 
\begin{equation*}
    L : = ((g+\sigma\p_x^4)G(0))^{\f12},
\end{equation*}
then $u = \psi -i LG(0)^{-1} \eta$ solves the linear dispersive equation
\begin{equation} \label{ToyModel}
    (\p_t + iL) u = P_{ext}.
\end{equation}
From Section $4.1$ of Micu and Zuazua \cite{zbMATH05046349}, we have the following \textbf{Ingham inequality} for a variant of \eqref{ToyModel}.
\begin{lemma}[\hspace{1sp}\cite{zbMATH05046349}] \label{t:Ingham}
For every $T>0$, there exists a positive constant $C_1(T)$ such that, for all $(w_n)_{n\in \Z} \in \ell^2(\Z, \C)$, $m\geq c$ for some constant $c>0$, one has the following Ingham inequality: 
Let $\mathfrak{t}(\xi)$ be a symbol that is defined in \eqref{SymbolTau}, then
\begin{equation*} 
    \int_0^T \Big| \sum_{n\in \N} w_n e^{-i\mathfrak{t}(n)t}\Big|^2 dt \geq C_1(T)\sum_{n\in \N} |w_n|^2.
\end{equation*}
\end{lemma}
The exact controllability of \eqref{ToyModel} can be established by the above Ingham inequality and the \textbf{Hilbert Uniqueness Method} (HUM).

Compared to the linearized system around the null solution, the main difficulties in proving the controllability of the full hydroelastic waves are:
\begin{enumerate}
\item The hydroelastic wave system is a quasi-linear system rather than a single dispersive equation with constant coefficients.
One cannot expect a simple step of symmetrization that turns \eqref{HydroWave} into \eqref{ToyModel}.

\item  The linearized system \eqref{e:ZeroLinear} does not have the sub-leading term of order $\f32$ and first-order transport term after symmetrization.

\item The classical Ingham inequality and harmonic analysis method only work for the linear dispersive equation with constant coefficients.
\end{enumerate}
In this paper, we overcome the above difficulties and prove the exact controllability of the full 2D hydroelastic waves.

\subsection{The main result}
Our main result asserts that, the hydroelastic wave system \eqref{HydroWave} is exactly controllable in arbitrary short time.
More precisely, we have the following result.
\begin{theorem}  \label{t:MainResult}
 Let $T>0$ be an arbitrary (possibly small) time and let $\omega\subset \T $ be a non-empty open set.
 There exist a positive universal constant $s_*$, and a small enough positive constant $\delta_*$ depending on $T$ and $\omega$ such that if $(\eta_{in}, \psi_{in}), (\eta_{end}, \psi_{end}) \in \H^{s_*}(\mathbb{T}; \R)$ with
 \begin{equation*}
     \|(\eta_{in}, \psi_{in})\|_{\H^{s_*}} +  \|(\eta_{end}, \psi_{end})\|_{\H^{s_*}} < \delta_* \,,
 \end{equation*}
 then there exists an exterior pressure $P_{ext}(t, x)$ supported on $[0,T]\times\omega$ with the following control properties:
 \begin{enumerate}  
\item The hydroelastic wave system \eqref{HydroWave} admits a unique solution
$$(\eta, \psi) \in C([0,T]; \H^{s_*})\cap C^1([0,T]; \H^{s_*-\f52})\cap C^2([0,T]; \H^{s_*-5}),$$ and $(\eta, \psi)|_{t=T} = (\eta_{end}, \psi_{end})$.
\item $P_{ext}(t, x)$ is in the function space $C([0,T]; H^{s_*})\cap C^1([0,T]; H^{s_*-\f52})\cap C^2([0,T]; H^{s_*-5})$, and satisfies the control estimates
\begin{align*}
   & \sum_{k=0}^2 \left(\|P_{ext}\|_{C^k([0,T]; H^{s_*-\f52 k})}  + \|(\eta, \psi)\|_{C^k([0,T]; \H^{s_*-\f52 k})}\right)  \lesssim    \|(\eta_{in}, \psi_{in})\|_{\H^{s_*}} +  \|(\eta_{end}, \psi_{end})\|_{\H^{s_*}}.
\end{align*}
 \end{enumerate}
\end{theorem}

Although one could in principle adapt the strategy of \cite{MR3776276} to establish null controllability for the system \eqref{HydroWave}, from which exact controllability would follow, we instead employ a direct Nash-Moser H\"ormander theorem approach, developed by Baldi, Haus and Montalto \cite{BHM} in the context of Hamiltonian nonlinear Schr\"odinger (NLS) equations.
As observed in \cite{WanY}, the system \eqref{HydroWave} can be rewritten as a quasilinear paradifferential equation of order $\frac{5}{2}$.
Compared to gravity–capillary water waves  of order $\f52$, the symbolic calculus and symmetrization for hydroelastic waves are substantially more involved and present additional technical challenges.
The Nash–Moser–H\"ormander approach avoids a direct reliance on paradifferential formulations and circumvents the need for a full paradifferential symbolic calculus.

The small-data assumption in Theorem  \ref{t:MainResult} is inherent to the quasilinear nature of the problem.
We expect that a suitable stabilization result for two-dimensional hydroelastic waves would allow the extension of controllability to large data, following the argument in Section 2.2 of \cite{MR3708590}.

\subsection{Ideas of the proof}
We now briefly outline the method for proving the controllability of 2D hydroelastic waves.
We rewrite the system \eqref{HydroWave} as 
\begin{equation} \label{e:hydro}
    \Phi (\mathbf{u}, f): =  \begin{pmatrix} P_1(\mathbf{u}) \\
    P_2(\mathbf{u})- \chi_\omega f \\
    \mathbf{u}(0) \\
     \mathbf{u}(T)
    \end{pmatrix} =  \begin{pmatrix} 0 \\
   0 \\
    \mathbf{u}_{in} \\
     \mathbf{u}_{end}
    \end{pmatrix},
\end{equation}
where $ \mathbf{u} = (\eta, \psi)$, $P_{ext} = \chi_\omega f$, and 
\begin{align*}
&P_1(\mathbf{u}) = \eta_t - G(\eta)\psi\,, \\
&P_2(\mathbf{u}) = \psi_t  + g\eta +\frac{1}{2} \psi_x^2- \f12\f{(\eta_x\psi_x-G(\eta)\psi)^2}{1+\eta^2_x} +\sigma \mathbf{E}(\eta).
\end{align*}
The linearized operator $\Phi'(\mathbf{u}, f )[\tilde{\mathbf{u}}, \tilde{f}]$ of $\Phi(\mathbf{u}, f)$ at $(\mathbf{u}, f)$ in the direction $(\tilde{\mathbf{u}}, \tilde{f})$ is
\begin{equation} \label{LinearDef}
\begin{aligned}
   \Phi'(\mathbf{u}, f )[\tilde{\mathbf{u}}, \tilde{f}] =  \begin{pmatrix} P_1'(\mathbf{u})[\tilde{\mathbf{u}}] \\
    P_2'(\mathbf{u})[\tilde{\mathbf{u}}]- \chi_\omega  \tilde{f} \\
    \tilde{\mathbf{u}}(0) \\
     \tilde{\mathbf{u}}(T)
    \end{pmatrix}.
\end{aligned}  
\end{equation}
We first obtain the existence of a right inverse of the linearized operator.
More precisely, we prove that, given any $(\mathbf{u}, f)$ and any $\mathbf{g} := (g_1, g_2, g_3, g_4)^T$ in suitable function spaces, there exist $(\tilde{\mathbf{u}}, \tilde{f})$ such that
\begin{equation}
    \Phi'(\mathbf{u}, f)[\tilde{\mathbf{u}}, \tilde{f}] =  \mathbf{g}.
\end{equation}
Moreover, we have to obtain appropriate estimates of $(\tilde{\mathbf{u}}, \tilde{f})$ in terms of $\mathbf{u}, f , \mathbf{g}$. 

To obtain the desired existence of the linearized control, we will first need to reduce and symmetrize the operator matrix $\mathbf{P}'(\mathbf{u}) = (P_1'(\mathbf{u}), P_2'(\mathbf{u}))^T$ in \eqref{PuTildeu} to differential operators with constant coefficients.
This is carried out in the following steps:
\begin{enumerate}
\item Conjugating $\mathbf{P}'(\mathbf{u})$ by the matrix $\mZ$ defined in \eqref{ZZeroDef}, so that by \eqref{ZConjugation}, one has $\mL_0 = \mZ^{-1} \mathbf{P}'(\mathbf{u}) \mZ $.

\item  Conjugating the matrix $\mL_0$ by the transformation $\mB$ defined in \eqref{BBInverseDef}, so that the matrix $\mL_0$ is reduced to $\mL_1$ in \eqref{LOneMatrix}.

\item  Conjugating the matrix $\mL_1$ by the transformation $\mA$ defined in \eqref{AAInverseDef}, so that the matrix $\mL_1$ is reduced to $\mL_2$ in \eqref{LTwoMatrix}.

\item Choosing matrices $P,Q$ in \eqref{PQDef}, then the matrix $\mL_2$ is reduced to $\mL_3$ in \eqref{LThreeMatrix}, which has constant coefficients at the leading $\f{5}{2}$ order terms.

\item Conjugating the matrix $\mL_3$ by the matrix operator $\mS$ defined in \eqref{SDef}, so that the matrix $\mL_4$ is symmetrized at the leading $\f{5}{2}$ order.

\item Conjugating the matrix $\mL_4$ by the matrix $M$ defined in \eqref{MDef}, so that the matrix $\mL_4$ is reduced to an anti-symmetric matrix $\mL_5$ up to order zero.
This also eliminates the sub-leading terms of order $\f32$.

\item Performing a translation of space variable $\mathcal{T}$ defined in \eqref{mTDef} to eliminate frequency zero component of order $1$ transport term.
The matrix $\mL_5$ is further reduced to $\mL_6$.

\item Conjugating the matrix $\mL_6$ by a unitary matrix $\mathcal{O}$ defined in \eqref{ODef}, so that the matrix $\mL_6$ becomes a diagonal matrix $\mL_7$ up to order zero.

\item  Using Fourier integral operators $A^{(1)}$ and $A^{(2)}$ in \eqref{UDef} to eliminate the order $1$ and $\f12$ terms in $\mL_7$.
We finally obtain the diagonal matrix $\mL_8$.
The remainder matrix $\mR_8$ is treated  perturbatively.
\end{enumerate}

Using these reductions, it is convenient for us to obtain the well-posedness and the observability of the linearized hydroelastic wave system by means of an Ingham-type inequality.
They can be obtained in a backward way in the above reduction steps.
Furthermore, one can obtain the controllability of the linearized system from the observability, the well-posedness of the Cauchy problem and HUM.
Finally, the controllability of the full nonlinear system can then be obtained from the linearized controllability by applying the \textbf{Nash-Moser-H\"ormander} Theorem \ref{t:Nash}.

We remark that using the above reduction procedure, one can also prove the local well-posedness of the two-dimensional hydroelastic waves \eqref{HydroWave} in the periodic setting as in the perturbation of Korteweg–de Vries (KdV) equation in \cite{BFH}.
However, this well-posedness result is at high regularity, and is much weaker compared to previous results in \cite{MR4104949, WanY}.
We do not prove this well-posedness result for the two-dimensional hydroelastic waves here.

The above symmetrization and reduction procedure for the linearized operator adapts the method used by Feola \& Procesi \cite{MR3360677} in the context of KAM theory for quasi-linear NLS equations. 
Similar reduction techniques were also  developed and used in the periodic setting for problems such as water waves, quasi-linear KdV, Benjamin-Ono, Kirchhoff, and incompressible Euler equations, see for instance \cite{MR2187619,MR3011291, MR3201904, MR3569244, MR3502158, MR3356988, MR3776276, MR4062430, MR3603787, MR4242904, MR4246389, MR4228858, MR4658635}.

Concerning the controllability for other quasi-linear dispersive PDEs except water waves, Baldi, Floridia \& Haus \cite{BFH,MR3711883} 
studied the internal controllability of quasi-linear perturbations of the KdV equation. 
Baldi, Haus \& Montalto \cite{BHM} and Iandoli \& Niu \cite{MR4700171} proved exact controllability for quasi-linear Hamiltonian Schrödinger equations in dimensions $d=1$ and $d\geq 2$ respectively.

The rest of this paper is organized as follows. 
In Section \ref{s:Linearized}, we compute the linearized hydroelastic wave system \eqref{PuTildeu}, in which the Dirichlet-Neumann operator is written as the leading contribution and the remainder term.
In the next two sections, we reduce the linearized hydroelastic wave system \eqref{PuTildeu} to the diagonal matrix operator $\mL_8$ with constant coefficients plus lower order remainder matrix $\mR_8$.
In Section \ref{s:Top}, the reduction is performed at the leading $\f52$ order and sub-leading $\f32$ order.
Then in Section \ref{s:Lower}, the reduction is performed until order $0$.
Using these reduction for the linearized hydroelastic waves, we prove the well-posedness of the Cauchy problem for the linearized hydroelastic waves in Section \ref{s:Well}, and obtain the observability of the linearized hydroelastic waves in Section \ref{s:Observe}.
As a consequence of above result, we prove in Section \ref{s:Control} the controllability of the linearized hydroelastic waves.
Finally, in Section \ref{s:Full}, we prove the controllability of the full hydroelastic waves by using the controllability of the linearized hydroelastic waves and the Nash-Moser-H\"ormander Theorem \ref{t:Nash}.
All auxiliary results are collected in Appendix \ref{s:Appendix}.
These include the definition of function spaces in the periodic setting, some estimates for the pseudo-differential operators, and the statement of the Nash-Moser-H\"ormander theorem.

\section{Linearized hydroelastic waves} \label{s:Linearized}
Recall that the hydroelastic waves \eqref{HydroWave} can be recast as \eqref{e:hydro}.
In this section, we compute the linearized operator $\Phi'(\mathbf{u}, f)[\tilde{\mathbf{u}}, \tilde{f}]$, and obtain an estimate for the second order derivative $\Phi''(\mathbf{u},f)[(\tilde{\mathbf{u}}_1, \tilde{f}_1), (\tilde{\mathbf{u}}_2, \tilde{f}_2)]$.

The linearized operator $\Phi'(\mathbf{u}, f )[\tilde{\mathbf{u}}, \tilde{f}]$ at the point $(\mathbf{u}, f)$ in the direction $(\tilde{\mathbf{u}}, \tilde{f})$ is given in \eqref{LinearDef}.
It suffices to compute $P_1'(\mathbf{u})[\tilde{\mathbf{u}}]$ and $P_2'(\mathbf{u})[\tilde{\mathbf{u}}]$.
We first recall the following result for the \textbf{shape derivative} of the Dirichlet-Neumann operator in Lannes \cite{MR2138139,MR3060183}.
\begin{lemma}[\hspace{1sp}\cite{MR2138139,MR3060183}]
Let $s> 1$, and $\psi \in H^\f32$.
Then the map $G(\cdot)\psi: H^{s+\f12} \rightarrow H^\f12$ is differentiable, and for any $\tilde{\eta}\in H^{s+\f12}$, 
\begin{equation} \label{ShapeDerivative}
G'(\eta)[\tilde{\eta}]\psi = \lim_{\varepsilon\rightarrow 0} \frac{1}{\varepsilon}\{G(\eta+\varepsilon\tilde{\eta})\psi - G(\eta)\psi\} = -G(\eta)(B\tilde{\eta}) - \p_x(V\tilde{\eta}),
\end{equation}
where $B$ and $V$ are the vertical and horizontal components of the velocity on $\Gamma_t$ given by
\begin{equation}\label{BVDef}
B(\eta,\psi) := \dfrac{\eta_x \psi_x + G(\eta)\psi}{1+\eta_x^2}, \quad V(\eta, \psi):=\psi_x - B\eta_x. 
\end{equation}
$B$ and $V$ satisfy the estimates, for $\tau > \f32$, 
\begin{equation*}
 \|B(\eta,\psi)\|_{H^{\tau-1}} + \|V(\eta,\psi)\|_{H^{\tau-1}} \leq C(\| \eta\|_{H^\tau})\|\psi\|_{H^\tau}.
\end{equation*}
\end{lemma}

We also have the following paralinearization for Dirichlet-Neumann operator.
\begin{lemma}[\hspace{1sp}\cite{MR3776276}] \label{t:DNParalinear}
Let $s \geq s_0$ with $s_0$ large enough.
Then there exists $\theta \in (0,1]$ such that 
\begin{equation*}
    G(\eta) = |D|(\psi - T_{B}\eta) - \p_x(T_V \eta) + F(\eta)\psi,
\end{equation*}
where $T_{B}\eta$ and $T_{V}\eta$ are the low-high paraproducts of $B\eta$ and $V\eta$, and the remainder term $F(\eta)\psi$ satisfies
\begin{equation} \label{DNRemainderEst}
\|F(\eta)\psi \|_{H^{s+\f32}} \leq C(\|\eta\|_{H^s}) \|\eta\|^\theta_{H^s} \|\psi\|_{H^s}.
\end{equation}
\end{lemma}

We then recall the estimates for second order derivative of the Dirichlet-Neumann operator from Proposition $3.25$ in Lannes \cite{MR2138139}.

\begin{lemma} [\hspace{1sp}\cite{MR2138139}] \label{t:Lannes}
Let $s\geq2$. Suppose that $\eta\in H^{s+\f32}\cap H^{\f92}$, $\psi_x\in H^{s+\f12}$. Then, for all $\tilde{\eta}_1,\tilde{\eta}_2\in H^{s+\f32}$,
\begin{align*}
\|\mathcal{G}''(\eta)[\tilde{\eta}_1,\tilde{\eta}_2]\psi\|_{H^{s+\f12}}\leq&C\left(s,\|\eta\|_{H^{\f92}},\|\psi_x\|_{H^{\f32}}\right)\Big(\|\tilde{\eta}_1\|_{H^{s+\f32}}\|\tilde{\eta}_2\|_{H^{\f52}}\\
&+\|\tilde{\eta}_2\|_{H^{s+\f32}}\|\tilde{\eta}_1\|_{H^{\f52}}+\|\tilde{\eta}_1\|_{H^{\f52}}\|\tilde{\eta}_2\|_{H^{\f52}}\left(\|\eta\|_{H^{s+\f32}}+\|\psi_x\|_{H^{s+\f12}}\right)\Big).
\end{align*}
\end{lemma}

According to the computation for Muskat problem with an elastic interface, in Lemma $3.1$ of \cite{WanY26}, the linearization of the elastic term is given by 
\begin{align*}
&\mathbf{E}'(\eta)[\tilde{\eta}]= \sum _{\al = 1}^4 E_{5-\al}(\eta)\p_x^{5-\al}\tilde{\eta} : = 
\f{\p_x^4\tilde{\eta}}{(1+\eta^2_x)^{\f52}}+2\big((1+\eta^2_x)^{-\f52}\big)_x\p^3_x\tilde{\eta}\\
&+\left\{((1+\eta^2_x)^{-\f52})_{xx}-\left(\f{5\eta_x\eta_{xx}}{(1+\eta_x^2)^{\f72}}\right)_x+\f{5(1-6\eta_x^2)\eta^2_{xx}}{2(1+\eta_x^2)^{\f92}}\right\}\tilde{\eta}_{xx}\\
&-\left[\left(\f{5\eta_x\eta_{xx}}{(1+\eta_x^2)^{\f72}}\right)_{xx}-\left(\f{5(1-6\eta_x^2)\eta^2_{xx}}{2(1+\eta_x^2)^{\f92}}\right)_x\right]\tilde{\eta}_{x}.
\end{align*}
It is important to note that $E_3(\eta) = 2 \p_x E_4(\eta)$.
Taking another derivative, we get that for $s \geq 2$,
\begin{equation} \label{SecondElastic}
 \|\mathbf{E}''(\eta)[\tilde{\eta}_1, \tilde{\eta}_2]\|_{H^s} \lesssim \|\tilde{\eta}_1 \|_{H^2}\|\tilde{\eta}_2 \|_{H^{s+4}} + \|\tilde{\eta}_1 \|_{H^{s+4}}\|\tilde{\eta}_2 \|_{H^2} + \|\tilde{\eta}_1 \|_{H^2}\|\tilde{\eta}_2 \|_{H^{2}} \|\eta\|_{H^{s+4}}.
\end{equation}

Hence, for $\mathbf{u} = (\eta, \psi)$, and $\tilde{\mathbf{u}} = (\tilde{\eta}, \tilde{\psi})$, using a direct computation that is similar to the computation in Section $5$ of \cite{MR3356988}, we obtain linearized operators
\begin{align*}
P_1'(\mathbf{u})[\tilde{\mathbf{u}}] =& \p_t \tilde{\eta} + \p_x(V\tilde{\eta}) - |D|(\tilde{\psi} - B\tilde{\eta}) + \mathcal{R}_G(\eta)\psi, \\
P_2'(\mathbf{u})[\tilde{\mathbf{u}}] =& \p_t \tilde{\psi} + V\p_x\tilde{\psi} - BG(\eta)\tilde{\psi} + (g+ BV_x)\tilde{\eta} + BG(\eta)(B \tilde{\eta}) \\
& +  \sigma \sum _{\al = 1}^4 E_{5-\al}(\eta)\p_x^{5-\al}\tilde{\eta},
\end{align*}
where $\mathcal{R}_G(\eta)\psi$ is a lower order remainder term that satisfies \eqref{DNRemainderEst}.
Alternatively, one can write these linearized operators in matrix form. 
\begin{equation} \label{PuTildeu}
 \begin{aligned}
 &\mathbf{P}'(\mathbf{u})[\tilde{\mathbf{u}}] = \mathbf{P}'(\mathbf{u})(\tilde{\eta}, \tilde{\psi})^T : = (P_1'(\mathbf{u})[\tilde{\mathbf{u}}], P_2'(\mathbf{u})[\tilde{\mathbf{u}}])^T\\
 =&  \begin{pmatrix}
   \p_t + V\p_x +V_x + |D|B &   -|D|+ \mathcal{R}_G(\eta) \\
    (g+BV_x)+ BG(\eta)B + \sigma \sum _{\al = 1}^4 E_{5-\al}(\eta)\p_x^{5-\al}&  \p_t + V\p_x - BG(\eta) 
    \end{pmatrix} \begin{bmatrix} \tilde{\eta} \\
    \tilde{\psi} \end{bmatrix}.
 \end{aligned}
\end{equation}

It is well-known that for water waves, instead of working with variables $(\eta, \psi)$, one may work with the \textbf{good unknown of Alinhac}, namely $(\eta, \psi - T_{B}\eta )$; see for example \cite{MR2138139, MR2805065}.
Motivated by this, $\mathbf{P}'(\mathbf{u})$ has the following conjugation structure:
\begin{equation} \label{ZConjugation}
 \mathbf{P}'(\mathbf{u}) = \mathcal{Z}\mathcal{L}_0 \mathcal{Z}^{-1},   
\end{equation}
where the matrix $\mathcal{Z}$ and its inverse $ \mathcal{Z}^{-1}$ are given by
\begin{equation} \label{ZDef}
 \mathcal{Z} : = \begin{pmatrix}
   1 &   0 \\
    B&  1 
    \end{pmatrix}, \quad \mathcal{Z}^{-1} : = \begin{pmatrix}
   1 &   0 \\
   -B&  1 
    \end{pmatrix},
\end{equation}
and the matrix $\mathcal{L}_0$ is given by
\begin{equation} \label{ZZeroDef}
 \mathcal{L}_0 : = \begin{pmatrix}
   \p_t + V\p_x + V_x  &   -|D|+ \mathcal{R}_G \\
   a + \sigma \sum _{\al = 1}^4 E_{5-\al}(\eta)\p_x^{5-\al} &  \p_t + V\p_x 
    \end{pmatrix}, \quad a = g+ B_t + VB_x.
\end{equation}
In the matrix $\mathcal{L}_0$, $a, B, V$ are periodic functions of $(\eta, \psi)$, and are thus periodic in space.
In addition, we have the bound
\begin{equation} \label{ZBoundInverse}
\|(\mZ -I)\tilde{\mathbf{u}} \|_{\H^s} + \|(\mZ^{-1} -I)\tilde{\mathbf{u}} \|_{\H^s} = \|B\tilde{\eta} \|_{H^{s+\f32}} \lesssim \| B\|_{_{H^{s+\f32}}}\|\tilde{\eta}\|_{H^{s+\f32}} \lesssim \|\psi \|_{_{H^{s+\f52}}}\|\tilde{\eta}\|_{H^{s+\f32}}.
\end{equation}

In the next two sections, we will further simplify the matrix operator $\mathcal{L}_0$.

\section{Reduction of the linearized operator at leading and sub-leading order} \label{s:Top}
In this section, we use several changes of variables to simplify the matrix operator $\mathcal{L}_0$ in \eqref{ZZeroDef}, so that the top order terms become constant-coefficient and symmetrized.
This is done using a change of space variable, a time reparametrization, a conjugation of matrices $P^{-1}$ and $Q$, and then the symmetrization of top order by conjugating $S$.
Then we symmetrize the matrix operator of lower orders.
The sub-leading terms of the matrix operator are eliminated at the same time.

\subsection{Change of space variable}
We start with an invertible time-dependent change of the space variable.
Consider a change of variables $y = x + \beta(t,x) \Leftrightarrow  x = y + \tilde{\beta}(t,y)$, where $\beta(t,x)$ is a periodic function in space with $|\beta_x|\leq \f12$.
Denote the transformation $\mB$ and its inverse $\mB^{-1}$ by
\begin{equation} \label{BBInverseDef}
    (\mathcal{B}h)(t,x) := h(t, x+ \beta(t,x)), \quad (\mathcal{B}^{-1}h)(t,y) := h(t, y+ \tilde{\beta}(t,y)).
\end{equation}
We will later choose the function $\beta$ in \eqref{beta}, so that using Moser's estimate \eqref{MoserTwo}, for each $t\in [0,T]$,
\begin{equation} \label{BetaHsEst}
 \|\beta(t,x)\|_{H^s} \lesssim \|(1+\eta_x^2)^\f12 -1\|_{H^{s-1}} \lesssim \|\eta \|_{H^s}, \quad s>1.
\end{equation}

For a given function $u$, we have the conjugation rule for $\mB$: $\mB^{-1} u \mB = (\mB^{-1} u)$. 
For the conjugation rule for $\mB$ on partial derivatives, we compute directly and use the rule $\mB\inv \p_x^{j+k} \mB = (\mB\inv \p_x^{j} \mB) (\mB\inv \p_x^{k} \mB)$.
\begin{align*}
\mB\inv \p_x \mB & = \{ \mB\inv(1 + \b_x) \} \p_y, \quad  
\mB\inv \p_{xx} \mB  = \{ \mB\inv(1 + \b_x) \}^2 \p_{yy} + (\mB\inv \b_{xx}) \p_y, 
\\
\mB\inv \p_{x}^3 \mB & = \{ \mB\inv(1 + \b_x) \}^3 \p_{y}^3 +3\{\mB\inv(1 + \b_x)\}(\mB\inv \b_{xx} )\p_{y}^2+(\mB\inv \p^3_x\b)\p_{y},
\\
\mB\inv \p_{x}^4 \mB & = \{ \mB\inv(1 + \b_x) \}^4 \p_{y}^4 +6\left\{\mB\inv(1 + \b_x)\right\}^2 (\mB\inv \b_{xx} )\p_{y}^3+4\left\{\mB\inv(1 + \b_x)\right\}(\mB\inv \p^3_x\b)\p^2_y\\
&+3(\mB\inv \b_{xx} )^2\p^2_y+(\mB\inv \p^4_x\b )\p_y,
\\
\mB\inv \p_t \mB & = \p_t + (\mB\inv \b_t) \p_y. 
\end{align*}
For the conjugation rule for $\mB$ on $|D_x| = \p_x \mH$, where $\mH$ is the Hilbert transform, we write
\begin{align*}
&\mB\inv |D_x| \mB 
= \mB\inv \p_x \mH \mB 
= (\mB\inv \p_x \mB) (\mB\inv \mH \mB)\\
=& \{ \mB\inv(1 + \b_x) \} \p_y [ \mH + (\mB\inv \mH \mB - \mH) ] = \{ \mB\inv(1 + \b_x) \} |D_y| + \mR_{\mB},   
\end{align*}
where $\mR_{\mB} := \{ \mB\inv(1 + \b_x) \} \p_y (\mB\inv \mH \mB - \mH)$ is bounded in time, regularizing in space, since by \eqref{BConjugationDiff} and \eqref{BetaHsEst}, we have for $s >0$,
\begin{align*}
 &\|\mR_{\mB} f \|_{C([0,T];H^{s+2})} \lesssim \left(1+\|\mB\inv \b_x  \|_{C([0,T];H^{s+2})}\right) \|\p_y^3 (\mB\inv \mH \mB - \mH) f \|_{H^s}\\
 \lesssim& \left(1+\|\mB\inv \b_x  \|_{C([0,T];H^{s+2})}\right)\bigl( \|\beta\|_{C([0,T];H^{5})} \|f\|_{C([0,T];H^s)} + \|\beta\|_{C([0,T];H^{s+5})} \|f\|_{C([0,T];L^2)} \bigr)\\
 \lesssim&  \|\eta\|_{C([0,T];H^{5})} \|f\|_{C([0,T];H^s)} + \|\eta\|_{C([0,T];H^{s+5})} \|f\|_{C([0,T];L^2)}.
\end{align*}

Hence, when conjugating the matrix operator $\mL_1$ by $\mB$, we obtain
\begin{equation} \label{LOneMatrix}
\mL_1 := \mB\inv \mL_0 \mB 
= \begin{pmatrix}
 \p_t + a_1 \p_y + a_2  \  & \ - a_3 |D_y| + R_1 \\ 
\sigma \sum_{j= 1}^4 a_{3+j}\p_y^{5-j} + a_8 \ & \  \p_t + a_1 \p_y
\end{pmatrix},    
\end{equation}
where coefficients are 
\begin{align*}
a_1 & := \mB\inv [  \b_t + V(1 + \b_x)], \quad 
a_2  := \mB\inv(V_x),  \quad
a_3  := \mB\inv(1 + \b_x), \quad 
a_4  := \mB\inv [E_4(1 + \b_x)^4], \\
a_5 & :=\mB\left[6E_4(1 + \b_x)^2\b_{xx}+E_3(1+\beta_x)^3\right],\\
a_6 & := \mB\inv\left[E_4\left(4(1+\beta_x)\p^3_x\beta+3\beta^2_{xx}\right)+3E_3(1+\beta_x)\beta_{xx}+E_2(1+\beta_x)^2\right],  \\
a_7 & :=\mB\inv\left[E_4\p_x^4\beta+E_3\p_x^3\beta+E_2\beta_{xx}+E_1(1+\beta_x)\right], \qquad 
a_8  := \mB\inv a,
\end{align*}
and $R_1 := - \mR_\mB - \mB\inv \mR_G \mB$.

To obtain bounds for the transformations $\mB$ and $\mB^{-1}$, we use Lemma \ref{t:BaldiLemma}, for $ s\geq 0$,
\begin{align}
 &\|\mB f \|_{C([0,T];H^s)}   \lesssim \|f\|_{C([0,T];H^s)} + \|\eta\|_{C([0,T];H^{s+2})}\|f\|_{C([0,T];L^2)},  \label{BBInverseBound}\\
 &\|\mB^{-1} f \|_{C([0,T];H^s)}   \lesssim \|f\|_{C([0,T];H^s)} + \|\eta\|_{C([0,T];H^{s+4})}\|f\|_{C([0,T];L^2)}.  \label{BBInvBoundTwo}
\end{align}
The operator $R_1$ is also regularizing in space, and for $s \geq s_0$
\begin{equation*}
\|R_{1} f \|_{C([0,T];H^{s+\f32})} \lesssim \|\eta\|^\theta_{C([0,T];H^{s+4})} \|f\|_{C([0,T];H^s)}.
\end{equation*}

\subsection{Time reparametrization and reduction at top order}
Consider a reparametrization of time in $[0,T]$,  $\tau = t + \al(t) \Leftrightarrow  t= \tau + \tilde{\al}(\tau)$, where $\al(t)$ satisfies that $\al(0)=\al(T)=0$.
Denote the transformation $\mA$ and its inverse $\mA^{-1}$ by
\begin{equation} \label{AAInverseDef}
    (\mathcal{A}h)(t,y) := h(t+\al(t), y), \quad (\mathcal{A}^{-1}h)(\tau,y) := h(\tau+ \tilde{\al}(\tau), y).
\end{equation}
We will later choose the function $\al(t)$ in \eqref{mu} and \eqref{alpha}.
One can check that
\begin{align*}
\mA\inv\p_y^k\mA=\p_y^k,\quad \mA\inv|D_y|\mA=|D_y|,\quad \mA\inv\p_t\mA=\{\mA\inv(1+\al')\}\p_{\tau}.
\end{align*}
Thus, we get the conjugation
\begin{equation} \label{LTwoMatrix}
\mL_2 :=\mA\inv\mL_1\mA=
 \begin{pmatrix}
 a_9\p_\tau+ a_{10} \p_y + a_{11}  \  & \ - a_{12}|D_y| + R_2 \\ 
\sigma \sum_{j= 1}^4 a_{12+j}\p_y^{5-j} + a_{17} \ & \  a_9\p_{\tau} + a_{10} \p_y
\end{pmatrix},
\end{equation}
where $R_2=\mathcal{A}^{-1}R_1\mA$, and  coefficients $a_i=a_i(\tau,y)$ are
\begin{align*}
a_9:=\mA\inv(1+\al'),\quad a_k:=\mA\inv(a_{k-9}),\quad k=10,...,17.
\end{align*}
$\mathcal{A}$ and $\mathcal{A}^{-1}$ preserve the Sobolev norm such that
\begin{equation} \label{AAInverseBound}
 \|\mathcal{A}h \|_{C([0,T]; H^s)} = \|\mathcal{A}^{-1}h \|_{C([0,T]; H^s)} = \|h \|_{C([0,T]; H^s)}, \quad \forall h \in C([0,T]; H^s).
\end{equation}

Given $b_1, b_2, b_3, b_4$ as functions of $(t,x)$, the system 
\[
\begin{pmatrix} 
b_1 & b_2 \\ 
b_3 & b_4 \end{pmatrix}
=
\begin{pmatrix} 
f & 0 \\ 
0 & g \end{pmatrix}
\begin{pmatrix} 
\lambda_1 & \lambda_2 \\ 
\lambda_3 & \lambda_4 \end{pmatrix}
\begin{pmatrix} 
p & 0 \\ 
0 & q \end{pmatrix}
\]
has solutions $f,g,p,q$, 
\begin{equation} \label{fgqDef}
f = \frac{b_1}{\lambda_1 p}\,, \quad 
g = \frac{b_3}{\lambda_3 p}\,, \quad
q = \frac{\lambda_1 b_2 p}{b_1 \lambda_2}\,, \quad
p = \text{any}\,,   
\end{equation}
if and only if $b_i, \lambda_i$ satisfy
\[
\frac{b_1 b_4}{b_2 b_3} \, = \frac{\lambda_1 \lambda_4}{\lambda_2 \lambda_3}\,.
\]
In particular, to have $\lambda_i$ constant  at the leading order, it is necessary that $b_1 b_4 / b_2 b_3$ is a constant at the leading order.

In order to reduce the coefficients of the leading terms of the matrix $\mL_2$ to constants, we use the above observation and look for functions $\al(t),\beta(t, x)$ such that
\begin{equation*}
m a^2_9=a_{12}a_{13},
\end{equation*}
for some constant $m\in\R$, which is equivalent to
\begin{equation*}
m(1+\al')^2=a_{3}a_{4}.
\end{equation*}
Since $a_3=\mB\inv(1+\beta_x)$, $a_4=\mB\inv[E_4(1+\beta_x)^4]$, and $\al'=\mathcal{B}^{-1}(\al')$, we then need
\begin{equation*}
m(1+\al')^2=E_4(1+\beta_x)^5,
\end{equation*}
that is, 
\begin{equation} \label{AlphaBetaEqn}
m^{\f15}(1+\al'(t))^{\f25}E_4(t,x)^{-\f15}=1+\beta_x.
\end{equation}
Integrating this equality in $dx$, and using the fact that $E_4(t,x) = (1+\eta_x^2)^{-\f52}$ we obtain that
\begin{align*}
1+\al'(t)=m^{-\f12}\left(\f{1}{2\pi}\int_0^{2\pi}(1+\eta_x^2)^{\f12}dx\right)^{-\f52}.
\end{align*}
Integrating the above equality in $dt$, we have
\begin{align}
&m=\left\{\f{1}{T}\int_0^{T}\left(\f1{2\pi}\int_0^{2\pi}(1+\eta_x^2)^{\f12}dx\right)^{-\f52}dt\right\}^2,  \label{mu} \\
&\a(t) =\int_0^t \left[ m^{-1/2} \left(\f{1}{2\pi}\int_0^{2\pi}(1+\eta_x^2)^{\f12}dx\right)^{-\f52} - 1 \right] dt.\label{alpha}
\end{align}
Set \[f(t)=\left(\f{1}{2\pi}\int_0^{2\pi}(1+\eta_x^2)^{\f12}dx\right)^{-\f52},\]
then from \eqref{mu}, one can write
\[1+\al'(t)=\f{f(t)}{\f1{T}\int_0^Tf(t)dt}.\]
Since $(1+\|\eta_x\|_{L^\infty})^{-\f52}\leq f(t)\leq 1$, we have a uniform bound 
\[
(1+\|\eta_x\|_{L^\infty_{t,x}})^{-\f52} \leq
1+\al'(t)\leq (1+\|\eta_x\|_{L^\infty_{t,x}})^{\f52}.\]

Using the equation \eqref{AlphaBetaEqn}, we get 
\begin{equation} \label{beta}
 \beta(t,x) =  \p_x\inv \left[m^{\f15}(1+\al'(t))^{\f25}(1+\eta_x^2)^{\f12} - 1 \right] , 
\end{equation}
where $\p_x\inv f$ is the primitive of $f$ in $x$, with zero spatial average.
Alternatively,
\[
\p_x\inv \, e^{ikx} = \frac{1}{ik} \, e^{ikx} \quad \forall k \in \Z \setminus \{ 0 \}, 
\qquad 
\p_x\inv \, 1 = 0.
\]

We follow the elementary observation above, with $\la_1=\la_2=\la_4=1$, $\la_3=m$, $p=1$, $b_1=a_9$, $b_2=a_{12}$, $b_3=a_{13}$, $b_4=a_9$, and set
\begin{equation} \label{PQDef}
 \begin{aligned}
&P:=\begin{pmatrix} 
a_9 & 0 \\ 
0 & a_{13}m^{-1}
\end{pmatrix},
\quad
P^{-1}=\begin{pmatrix} 
a^{-1}_9 & 0 \\ 
0 & a^{-1}_{13}m
\end{pmatrix},\\
&Q:=\begin{pmatrix} 
1 & 0 \\ 
0 & a^{-1}_{13}a_9
\end{pmatrix},
\quad
Q^{-1}=\begin{pmatrix} 
1 & 0 \\ 
0 & a_{13}a_9^{-1}
\end{pmatrix}.
\end{aligned}   
\end{equation}
According to \eqref{fgqDef}, we then have
\begin{equation}  \label{LThreeMatrix}
\mL_3:=P^{-1}\mL_2 Q
=\begin{pmatrix}
 \p_\tau+ a_{18} \p_y + a_{19}  \  & \ - |D_y|+a_{20}\mH+R_3 \\ 
\sigma m \p^4_y+m\sigma\sum_{j= 2}^4 a_{19+j}\p_y^{5-j} + m a_{24} \ & \  \p_{\tau} + a_{18} \p_y+a_{25}
\end{pmatrix},
\end{equation}
where coefficients are 
\begin{align*}
&a_{18}=\f{a_{10}}{a_9},\quad a_{19}=\f{a_{11}}{a_9},\quad a_{20}=-\f{a_{12}}{a_{9}}\left(\f{a_{12}}{a_9}\right)_y,\\
&a_{19+j}=\f{a_{12+j}}{a_{13}},\, j=2,3,4,\quad a_{25}=m\f{a_9}{a_{13}}\left(\f{a_9}{a_{12}}\right)_{\tau}+m\f{a_{10}}{a_{13}}\left(\f{a_9}{a_{12}}\right)_y,
\end{align*}
and the remainder term $R_{3}$ is given by
\begin{align*}
R_{3}=-\f{a_{12}}{a_9}\p_y\left[\mH, \f{a_9}{a_{12}}\right]+\f{1}{a_{9}}R_2\f{a_9}{a_{12}}.
\end{align*}
Note that the leading order operators in matrix $\mL_3$, namely $-|D_y|$ and $\sigma m \p_y^4$ now have constant coefficients.

Recall that by definition, constants $a_9$, $a_{13}$ and $a_{13}$ are given by
\begin{equation*}
a_9 = 1+\mA^{-1}\al', \quad a_{12} = m^\f15\mA^{-1}\mB\inv[(1+\al'(t))^{\f25}(1+\eta_x^2)^{\f12}], \quad 
a_{13}  = m^\f45\mA^{-1}\mB\inv [(1+\al'(t))^{\f85}(1+\eta_x^2)^{-\f12}].
\end{equation*}
Using Moser's estimate \eqref{MoserTwo}, we get that for $s>2$,
\begin{equation}\label{L2inv}
  \|(P^{\pm1} - \text{diag}\{1, m^{\pm\f15} \}) \tilde{u}\|_{\H^s} +  \|(Q^{\pm 1}- \text{diag}\{1, m^{\mp\f15} \})\tilde{u}\|_{\H^s} \lesssim \|\eta\|_{H^{2}}\|\tilde{u}\|_{\mathcal{H}^s} + \|\eta\|_{H^{s+2}}\|\tilde{u}\|_{\mathcal{H}^2}.
\end{equation}
Using the commutator estimate \eqref{HilbertCommutator}, the operator $R_3$ satisfies the estimate, for $s \geq s_0$,
\begin{equation*}
\|R_3 f \|_{C([0,T];H^{s+\f32})} \lesssim \|\eta\|^\theta_{C([0,T];H^{s+4})} \|f\|_{C([0,T];H^s)}.
\end{equation*}

\subsection{Symmetrization of top order}
Let $\mathfrak{g}: \R \rightarrow \R^+$ be a $C^\infty$ function such that
\begin{equation} \label{LambdaSymbolDef}
\mathfrak{g}(\xi) = \left( \f{g+ \sigma \xi^4}{|\xi|}\right)^\f12, \quad \forall |\xi| \geq \f23; \quad \mathfrak{g}(\xi) = 1, \quad \forall |\xi| \leq \f13.
\end{equation}
We define $\Lambda$ as the Fourier multiplier with symbol $\mathfrak{g}$, and 
\begin{equation} \label{SDef}
 S:=\begin{pmatrix} 
1 & 0 \\ 
0 & m^{\f12}\Lambda
\end{pmatrix},
\quad
S^{-1}=\begin{pmatrix} 
1 & 0 \\ 
0 & m^{-\f12}\Lambda^{-1}
\end{pmatrix},
\end{equation}
where the $m$ is constant in \eqref{mu}.

Writing the matrix $ \mL_3 = \begin{pmatrix} A_3 & B_3 \\ 
C_3 & D_3 \end{pmatrix}$,  we then have the conjugation 
\begin{equation*}
 S\inv \mL_3 S
= \begin{pmatrix} 
A_3 & m^{1/2} B_3 \Lambda
 \\
m^{-1/2} \Lambda^{-1} C_3 & \Lambda^{-1} D_3 \Lambda
\end{pmatrix}
=: \begin{pmatrix} A_3^+ & B_3^+ \\ 
C_3^+ & D_3^+ \end{pmatrix}.
\end{equation*}

We define the operator 
\begin{equation} \label{TDef}
    T : = \sqrt{m}|D_y|^\f12 (g+\sigma \p_y^4)^\f12.
\end{equation}
Using the symbolic expansion \eqref{SymbolExpansion}, we get that $A_3^+ = A_3$ and 
\begin{align*}
B_3^+ &= -T+ \sqrt{m}a_{20} \mH |D_y|^{-\f12} (g+\sigma \p_y^4)^\f12 + \sqrt{m} R_3 \Lambda,\\
C_3^+ &=  T+ \sqrt{m}\sigma a_{21}\Lambda^{-1} \p_y^3 +  \sqrt{m}\sigma (a_{21})_y (\Lambda^{-1})_1 \p_y^3 + \sqrt{m}a_{22} \Lambda^{-1}\p_y^2 +g\sqrt{m}\pi_0+\mR_{3,C}^+,\\
D_3^+& =  \p_\tau + a_{18}\p_y + (a_{18})_y(\Lambda^{-1})_1 \Lambda \p_y + a_{25} + \mR_{3,D}^+,
\end{align*}
where $(\Lambda^{-1})_n : =\frac{1}{i^n}$Op$(\p_\xi^n(1/\mathfrak{g}))$ is the $n$-th expansion of $\Lambda^{-1}$, and $\pi_0(h): = \frac{1}{2\pi}\int_\T h dy$ is the spatial average. 
We further expand
\begin{align*}
 (g+\sigma \p_y^4)^\f12 &=  \sqrt{\sigma} \p_y^2 + \sqrt{g}\pi_0 + \mR_\sigma, \quad
 \Lambda^{-1} = \f{1}{\sqrt{\sigma}} |D_y|^{-\f32} + O(|D_y|^{-\frac{11}{2}}), \\
 (\Lambda^{-1})_1 &= -\f{3}{2\sqrt{\sigma}} |D_y|^{-\f52}\mH + O(|D_y|^{-\frac{11}{2}}),
\end{align*}
where each remainder term of $O(|D_y|^\al)$ type is a Fourier multiplier whose symbol $g(\xi)$ admits the bound $|g(\xi)| \leq C(1+ |\xi|)^\al$ for all $\xi \in \R$ and some $C>0$ depending only on $\sigma$.

Hence, we get using $\p_y = -|D_y|\mH$,
\begin{equation} \label{SLThreeConjugation}
S\inv \mL_3 S
= \mL_4 + \mR_4 := \begin{pmatrix} A_4 & B_4 \\ 
C_4 & D_4 \end{pmatrix} +  \begin{pmatrix} 0 & \mR_{4,B} \\ 
\mR_{4,C} & \mR_{4,D} \end{pmatrix},
\end{equation}
for 
\begin{align*}
&A_4 := \p_\tau+ a_{18} \p_y + a_{19}, \quad B_4 := -T - \sqrt{m\sigma} a_{20} |D_y|^\f{3}{2}\mH, \\
&C_4 := T + \sqrt{m\sigma} a_{21} |D_y|^\f{3} {2}\mH +a_{26}|D_y|^{\f12}, \quad D_4 := \p_\tau+ a_{18} \p_y + a_{27},
\end{align*}
and $\mR_4$ is a smoothing matrix that absorbs all remainder terms.
It satisfies the bound 
\begin{equation*}
\|\mR_4 \mathbf{f} \|_{C([0,T];H^{s}\times H^s)} \lesssim \|\eta\|^\theta_{C([0,T];H^{s+4})} \|\mathbf{f}\|_{C([0,T];H^{s}\times H^s)}, \quad s \geq s_0.
\end{equation*}
The coefficients $a_{26}$ and $a_{27}$ are given by
\begin{equation*}
 a_{26} = \f32 \sqrt{m\sigma} (a_{21})_y - \f{\sqrt{m}}{\sqrt{\sigma}}a_{22} ,\quad a_{27} =  a_{25} + \f32 (a_{18})_y.
\end{equation*}

\subsection{Symmetrization of lower orders}
Next, we further symmetrize and rewrite the linearized operator, so that the sub-leading terms of order $\f32$ are eliminated.
To achieve this goal, we symmetrize the matrix operator $\mL_4 + \mR_4$ up to order zero.
Let us consider the symmetrized matrix operator 
\begin{equation}  \label{symm mL5}
\mL_5 := \begin{pmatrix} 
A_5 &  - C_5 \\ 
C_5 & A_5 \end{pmatrix}, 
\qquad 
\begin{array}{l}
A_5 :=  \p_\tau + a_{18} \p_y + a_{28} ,  
 \\
C_5 := T + a_{29} \mH |D_y|^{3/2} + a_{30} |D_y|^{1/2},
\end{array}
\end{equation}
where $a_{28}$, $a_{29}$ and $a_{30}$ are real-valued spatial-periodic functions of $(\tau, y)$.
We consider a transformation 
\begin{equation} \label{MDef}
 M = \begin{pmatrix} 
1 &  \ell \\ 
0 & v \end{pmatrix}, 
\qquad 
\begin{array}{l}
\ell = \ell_1  |D_y|^{-5/2},  
 \\
 v =  1 + v_1 \mH |D_y|^{-1}+  v_2 |D_y|^{-2},
\end{array}
\end{equation}
where $v_1$, $v_2$ and $\ell_1$  are real-valued spatial-periodic functions of $(\tau, y)$.
Using the formula \eqref{SymbolExpansion}, one can expand $|D_y|^s a$,
\begin{equation*}
|D_y|^s a = a |D_y|^s + s a_y |D_y|^{s-1} \mH 
- \frac{s(s-1)}{2} a_{yy} |D_y|^{s-2} + O(|D_y|^{s-3}).
\end{equation*}
We compute the matrix
\begin{equation*}
 \mL_4 M -M\mL_5 = \begin{pmatrix} 
A_4 -A_5 -\ell C_5 &  B_4 v + C_5 + A_4 \ell - \ell A_5 \\ 
C_4 - v C_5 & D_4 v - vA_5 + C_4 \ell \end{pmatrix}.
\end{equation*}
We will choose the coefficients so that each component of the above matrix is $O(|D_y|^{-\f12})$.

For the position $(1.1)$, we need to have
\begin{equation}\label{eq1}
a_{19} - a_{28} = \sqrt{m\sigma} \ell_1.
\end{equation}
For the position $(1,2)$, we compute
\begin{align*}
B_4 v &= -T + \sqrt{m\sigma}(v_1-a_{20}) \mH|D_y|^\f32 + \sqrt{m\sigma}\Big(\f52(v_1)_y - v_2 + a_{20}v_1\Big)|D_y|^\f12 + O(|D_y|^{-\f12}),  \\
A_4 \ell &= \ell \p_\tau + O(|D_y|^{-\f12}),\quad
\ell A_5  = \ell\p_\tau + O(|D_y|^{-\f12}).
\end{align*}
To eliminate the $O(|D_y|^{\f32})$ term and $O(|D_y|^{\f12})$ term, we must have
\begin{align}
  \sqrt{m\sigma}(v_1-a_{20}) = - a_{29},\label{eq2} \\
 \sqrt{m\sigma} \left(\f52(v_1)_y - v_2 + a_{20}v_1 \right) = - a_{30}\label{eq3}.
\end{align}
For the position $(2,1)$, we compute 
\begin{equation*}
    vC_5 = T+(a_{29}+\sqrt{m\sigma}v_1)|D_y|^{\f32}\mH+(a_{30}-a_{29}v_1+\sqrt{m\sigma}v_2)|D|^{\f12}+O(|D_y|^{-\f12}).
\end{equation*}
To eliminate the $O(|D_y|^{\f32})$ term and $O(|D_y|^{\f12})$ term, we must have the equation \eqref{eq2} and 
\begin{equation} \label{eq4}
(a_{30}-a_{29}v_1+\sqrt{m\sigma}v_2)=a_{26}.
\end{equation}
For the position $(2,2)$, we have
\begin{align*}
D_4v&= v\p_{\tau}-a_{18}\mH|D_y|+(a_{18}v_1+a_{27})+O(|D_y|^{-\f12}),\\
vA_5&= v\p_{\tau}-a_{18}\mH|D_y|+(v_1a_{18}+a_{28}) +O(|D_y|^{-\f12}),\\
C_4\ell&=\sqrt{\sigma m}\ell_1+O(|D_y|^{-\f12}).
\end{align*}
To eliminate the $O(1)$ term, we must have
\begin{equation}\label{eq5}
    a_{27}- a_{28} + \sqrt{\sigma m}\ell_1 = 0.
\end{equation}

\eqref{eq1}-\eqref{eq5} is a system of $5$ equations with $6$ unknowns $v_1,\,v_2,\,\ell_1,\,a_{28},\,a_{29},\,a_{30}$.
To better simplify the linearized operator, we set $a_{29} = 0$, so that other unknowns can then be determined, and the order $\f32$ term in the matrix $\mL_5$ are eliminated.
In \cite{BHM}, the sub-leading terms are eliminated by conjugating a multiplication function.
In our analysis, this step is integrated in the symmetrization.

From \eqref{eq1} and \eqref{eq5}, we can obtain
\begin{equation}
a_{28}=\f{a_{19}+a_{27}}{2},\quad \ell_1=\f{a_{19}-a_{27}}{2\sqrt{m\sigma}}.
\end{equation}
Since $a_{29}=0$, from \eqref{eq2}, we obtain
\begin{equation}
v_1=a_{20}.
\end{equation}
Thus, \eqref{eq3} and \eqref{eq4} become
\begin{align*}
\sqrt{m\sigma}v_2 - a_{30}=&\left(\f52(a_{20})_y + a_{20} \right),\\
\sqrt{m\sigma}v_2+a_{30}=&a_{26},
\end{align*}
which gives that
\begin{equation}
v_2=\f1{2\sqrt{m\sigma}}\left(\f52(a_{20})_y +a_{20}+a_{26}\right),\quad a_{30}=\f12\left(a_{26}-\f52(a_{20})_y -a_{20}\right).
\end{equation}
Therefore, we have determined the matrix $M$. 
Note that $M$ can be written as the identity matrix plus a small perturbation:
\begin{equation*}
    M = I + M_R := I +  \begin{pmatrix} 
0 &  \ell_1  |D_y|^{-5/2} \\ 
0 & v_1 \mH |D_y|^{-1}+  v_2 |D_y|^{-2} \end{pmatrix},
\end{equation*}
Note that $\ell_1$, $v_1$ and $v_2$ are small in the sense that 
\begin{equation*}
 \| \ell_1\|_{C([0,T]; H^s)} + \| v_1\|_{C([0,T]; H^s)} + \| v_2\|_{C([0,T]; H^s)} \lesssim \|(\eta, \psi) \|_{C([0,T]; \H^{s+2})} \ll 1.
\end{equation*}
The matrix $M$ is then invertible with inverse $M^{-1} = I + \sum_{k=1}^\infty (-1)^k M_R^k$.
We obtain the estimate
\begin{equation} \label{MpmBound}
   C_1 \| \mathbf{f}\|_{C([0,T]; H^s\times H^s)} \leq \|M^{\pm 1} \mathbf{f}\|_{C([0,T]; H^s\times H^s)} \leq C_2\| \mathbf{f}\|_{C([0,T]; H^s\times H^s)},
\end{equation}
for some constants $C_2>C_1 >0$.

Moreover, we have the conjugation
\begin{equation} \label{SLFourConjugation}
M\inv (\mL_4 + \mR_4) M
= \mL_5 + \mR_5, \quad \mR_5 : =  M\inv \mR_4 M + (M\inv \mL_4  M - \mL_5).
\end{equation}
 $\mR_5$ is a smoothing matrix that absorbs all remainder terms with order less than $0$.
It satisfies the bound, for $s \geq s_0$,
\begin{equation*}
\|\mR_5 \mathbf{f} \|_{C([0,T];H^{s}\times H^s)} \lesssim \|\eta\|^\theta_{C([0,T];H^{s+4})} \|\mathbf{f}\|_{C([0,T];H^{s}\times H^s)}.
\end{equation*}

\section{Reduction of the linearized operator at lower orders} \label{s:Lower}
In this section, we further reduce and simplify the linearized operator at order $1$ and $\f12$.
First, we use a translation of space variable $z = y + p(\tau)$ to remove the space average of the coefficient $a_{18}(\tau, y)$ of the transport term in $\mL_5$.
Then we construct a Fourier integral operator to eliminate the order $1$ and $\f12$ terms.
\subsection{Translation of space variable}
We consider the operator $\mathcal{T}$ and its inverse:
\begin{equation} \label{mTDef}
\mathcal{T} h(\tau,y):=h(\tau,y+p(\tau)),\quad \mathcal{T} ^{-1} h(\tau,z):=h(\tau,z-p(\tau)).
\end{equation}
$\mathcal{T}$ and $\mathcal{T}^{-1}$ preserve the Sobolev norm in the sense that
\begin{equation} \label{TTInverseBound}
 \|\mathcal{T}h \|_{C([0,T]; H^s)} = \|\mathcal{T}^{-1}h \|_{C([0,T]; H^s)} = \|h \|_{C([0,T]; H^s)}, \quad \forall h \in C([0,T]; H^s).
\end{equation}

One can check that $\mathcal{T}^{-1}T(D_y)\mathcal{T}=T(D_z)$, $\mathcal{T}^{-1}\p_y\mathcal{T}=\p_z$ and $\mathcal{T}^{-1}\p_\tau\mathcal{T}=\p_\tau+p'(\tau)\p_z$. Thus, we have
\begin{equation*}
\mL_6:=\mathcal{T}^{-1}\mL_5 \mathcal{T}
=\begin{pmatrix}
 \p_\tau + a_{31}\p_z + a_{32}  \  & \ -T -a_{33} |D_z|^{1/2} \\ 
T  + a_{33}|D_z|^{1/2}\ & \  \p_\tau + a_{31} \p_z+ a_{32}
\end{pmatrix},
\end{equation*}
where 
\begin{align*}
a_{31}(\tau,z)=\mathcal{T}^{-1}a_{18}(\tau,z)+p'(\tau),\quad a_{32}(\tau,z)=\mathcal{T}^{-1}a_{28}(\tau,z),\quad a_{33}(\tau,z)=\mathcal{T}^{-1}a_{30}(\tau,z).
\end{align*}
Choosing 
\begin{equation*}
p(\tau)=-\f{1}{2\pi}\int_0^\tau\int_0^{2\pi}a_{18}(s,y)\,dyds,
\end{equation*}
then the transport coefficient $a_{31}$ has zero space average:
\begin{equation*}
\int_0^{2\pi}a_{31}(\tau,z)\,dz=0.
\end{equation*}
Therefore, we get that
\begin{equation} \label{TLFiveConjugation}
\mathcal{T}\inv (\mL_5 + \mR_5) \mathcal{T}
= \mL_6 + \mR_6, \quad \mR_6 : =  \mathcal{T}\inv \mR_5 \mathcal{T}.
\end{equation}

We further choose unitary matrices $\mathcal{O}$ and $\mathcal{O}^{-1}$
\begin{equation} \label{ODef}
\mathcal{O} = \frac{1}{\sqrt{2}} \begin{pmatrix}
 1    & 1 \\ 
i &  -i
\end{pmatrix}, \quad \mathcal{O}^{-1} = \frac{1}{\sqrt{2}} \begin{pmatrix}
 1    & -i \\ 
1 &  i
\end{pmatrix}.
\end{equation}
Then we have the conjugation
\begin{equation*}
\mathcal{L}_7 :=\mathcal{O}^{-1}\mL_6 \mathcal{O}
=\begin{pmatrix}
 \p_\tau +iT+a_{31}\p_z +ia_{33}|D_z|^\f12+ a_{32}    & 0 \\ 
0 &  \p_\tau -iT+a_{31}\p_z -ia_{33}|D_z|^\f12+ a_{32}
\end{pmatrix}.
\end{equation*}
$\mathcal{L}_7$ is a diagonal matrix.
We also define the remainder matrix $\mR_7 = \mathcal{O}^{-1}\mR_6 \mathcal{O}$.

\subsection{Reduction at lower orders}
We now use a Fourier integral operator $A$ to eliminate the coefficients $a_{31}$ and $a_{33}$.
Given a spatial periodic function $h(\tau, z)$, we define the Fourier integral operator $A^{(1)}$ as
\begin{equation} \label{ADef}
h(\tau,z) = \sum_{j \in \Z} h_j(\tau) \, e^{ijz} 
\quad \mapsto \quad
A^{(1)}h(\tau, z) = \sum_{j \in \Z} h_j(\tau) \, p(\tau,z,j) \, e^{i \phi(\tau,z,j)},
\end{equation}
where the amplitude $p(\tau,z,j)$ is a spatial-periodic symbol of order zero,
the phase function $\phi(\tau,z,j)$ is of the form
\begin{equation} \label{PhaseFunction phi}
\phi(\tau,z,j) = j z + |j|^{1/2} \gamma(\tau,z),
\end{equation}
and $\gamma(\tau,z)$ is a spatial-periodic function, with $|\gamma_z(\tau,z)| < 1/2$. 
We consider the operator 
\begin{equation} \label{DDefLseven}
 \mathcal{D}_+ : = \p_\tau + iT, \quad \mL_7^{(1)}: = \p_\tau +iT+a_{31}\p_z +ia_{33}|D_z|^\f12+ a_{32}.
\end{equation}
In the following, we will show the existence of functions $p(\tau, z,j)$ and $\gamma(\tau, z)$ such that $\mL_7^{(1)}A^{(1)} - A^{(1)} \mD_+ = O(|D_z|^{-\f12})$.
Let $\mathfrak{t} : \R \to \R$ be a $C^\infty$ function such that 
\begin{equation}  \label{SymbolTau}
\mathfrak{t}(\xi) = \{ m |\xi| (1+\sigma \xi^4)\}^{1/2}, 
\quad \forall |\xi| \geq 2/3; 
\qquad 
\mathfrak{t}(\xi) = 0, 
\quad \forall |\xi| \leq 1/3,
\end{equation}
so that Op$(\mathfrak{t}) = T$ on the periodic functions. 
Using Lemma \ref{t:Composition}, we rewrite the commutator $[T,A^{(1)}]$ as
\begin{equation*}
 [T,A^{(1)}] = \sum_{n=1}^4 \f{1}{n!}\text{Op}(\p_z^n a) \circ \text{Op}(i^{-n}\p_\xi^n \mathfrak{t}) + \mR_{A1},\quad  \mR_{A1}  = O(1),
\end{equation*}
where $a$ is the symbol
\begin{equation*}
   a(\tau, z,j) = p(\tau,z,j) \, e^{i |j|^{1/2}\gamma(\tau,z)},
\end{equation*}
and $\mR_{A1}$ is the lower-order remainder term of the commutator.
We compute
\begin{align*}
\text{Op}(i^{-1}\p_{\xi}\mathfrak{t})
=&\f52\sqrt{m\sigma}|D_z|^{\f32}\mH+O(|D_z|^{-\f52}),\\
\text{Op}(i^{-2}\p_{\xi}^2\mathfrak{t}) 
=&-\f{15}{4}\sqrt{m\sigma}|D_z|^{\f12}+O(|D_z|^{-\f{7}{2}}),\\
\text{Op}(i^{-3}\p_{\xi}^3\mathfrak{t})
=&-\f{15}{8}\sqrt{m\sigma}|D_z|^{-\f12}\mH+O(|D_z|^{-\f{9}{2}}),\\
\text{Op}(i^{-4}\p_{\xi}^4\mathfrak{t})
=&-\f{15}{16}\sqrt{m\sigma}|D_z|^{-\f32}+O(|D_z|^{-\f{11}{2}}),\\
\p_za=&\{i|j|^{\f12}p\gamma_z+p_z\}e^{i|j|^{\f12}\gamma},\\
\p^2_za=&\{-i|j|p\gamma^2_z+i|j|^{\f12}(2p_z\gamma_z+p\gamma_{zz})+p_{zz}\}e^{i|j|^{\f12}\gamma},\\
\p^3_za=&\{-i|j|^{\f32}p\gamma^3_z-3|j|\beta_z(p_z\gamma_z+p\gamma_{zz})+r_1\}e^{i|j|^{\f12}\gamma},\\
\p^4_za=&|j|^2p\gamma^4_ze^{i|j|^{\f12}\gamma}+r_2,
\end{align*}
where $r_1$ and $r_2$ are functions satisfying
\begin{align*}
\|r_k(\cdot,\cdot,j)\|_{H^s}
\leq (1+|j|)^{k-\f12}\left(\|\gamma\|_{H^{s+2}}+\|p-1\|_{H^{s+2}}\right),\quad\forall s\geq1,\quad k=1,2,\quad j\in\Z.
\end{align*}
We also have
\begin{align*}
\p_zA^{(1)}h=\sum_{j\in\Z}h_j(ijp+i|j|^{\f12}\gamma_zp+p_z)e^{i\phi(\tau,z,j)}.
\end{align*}
From Lemma 12.10 of \cite{MR3356988}, we get
\begin{align*}
|D_z|^{\f12}A^{(1)}h=\sum_{j\neq0}h_j|j|^{\f12}p(\tau,z,j)e^{i\phi(\tau,z,j)}+R_{A2}h,\quad R_{A2}=O(1).
\end{align*}
The commutator $[\p_\tau,A^{(1)}]=\p_{\tau}A^{(1)}-A^{(1)}\p_{\tau}$ is
\begin{align*}
[\p_{\tau},A^{(1)}]h=\sum_{j\in\Z}h_j\left(p_{\tau}(\tau,z,j)+i|j|^{\f12}\gamma_\tau(\tau,z)p(\tau,z,j)\right).
\end{align*}
Using the expansions above, we have
\begin{align*}
\mL_7^{(1)}A^{(1)}-A^{(1)}\mathcal{D}_+=E+\mathcal{R}_{A3}, \quad \mathcal{R}_{A3}:=i\mathcal{R}_{A1}+ia_{33}\mathcal{R}_{A2},
\end{align*}
and $E$ is the operator 
\begin{align*}
Eh=\sum_{j\in\Z}h_j(t)c(\tau,z,j)e^{i\phi(\tau,z,j)},
\end{align*}
with
\begin{align*}
&c(\tau,z,j)=i\f52\sqrt{m\sigma}\text{sign}(j)\gamma_z p|j|^2 +\left(-\f{15}{8}\sqrt{m\sigma}\gamma^2_zp+\f52\sqrt{m\sigma}\text{sign}(j)p_z\right)|j|^{\f32}\\
&+\left(i\f{15}{48}\sqrt{m\sigma}\text{sign}(j)\gamma_z^3p-i\f{15}{4}\sqrt{m\sigma}\gamma_zp_z-i\f{15}{8}\sqrt{m\sigma}\gamma_{zz}p+ia_{31}\text{sign}(j)p\right)|j|\\
&+\left(-i\f{15}{8}\sqrt{m\sigma}p_{zz}+\f{15}{16}\sqrt{m\sigma}\text{sign}(j)\gamma_z(p_z\gamma_z+p\gamma_{zz})-\f{5}{2^7}p\gamma_z^4
+i\gamma_{\tau}p+ia_{33}p+ia_{31}\gamma_zp\right)|j|^{\f12}\\
&+r(\tau, z, j),
\end{align*}
where $r(\tau, z, j)$ is the remainder term of order less than or equal to zero.
We now need to choose $\gamma$ and $p$ such that $c(\tau, z, j)$ is a symbol of order less than or equal to zero. 

For the order $2$ term in $c(\tau,z,j)$, we need to have $\gamma_z p =0$.
Since we cannot have $p = 0$, it turns out that $\gamma_z = 0$, and
\begin{equation*}
    \gamma(\tau, z) = \gamma_0(\tau),
\end{equation*}
for some function $\gamma_0(\tau)$ that only depends on $\tau$ to be determined later.
We further expand symbol $p(\tau, z, j )$ as
\begin{equation*}
    p(\tau, z, j ) = \sum_{-2 \leq m \leq 0}|j|^{m/2}p^{(m)}(\tau, z, j ), \quad p^{(m)}\text{ is bounded in } j.
\end{equation*}
Using the fact that $\gamma_z = 0$, we rewrite $c(\tau, z, j)$ as
\begin{align*}
 &c(\tau, z, j) =   \left(\f52\sqrt{m\sigma}\text{sign}(j)p_z^{(0)}\right)|j|^{\f32} + \left(\f52\sqrt{m\sigma}\text{sign}(j)p_z^{(-1)} +ia_{31}\text{sign}(j)p^{(0)}\right)|j| \\
 +&  \left(\f52\sqrt{m\sigma}\text{sign}(j)p_z^{(-2)} +ia_{31}\text{sign}(j)p^{(-1)} -i\f{15}{8}\sqrt{m\sigma}p_{zz}^{(0)}+i\gamma_{\tau}p^{(0)}+ia_{33}p^{(0)}\right)|j|^{\f12} + \tilde{r}(\tau, z, j),
\end{align*}
where $\tilde{r}(\tau, z, j)$ is the remainder term of order less than or equal to zero.

For the order $\f32$ term in $c(\tau,z,j)$, we need to have $\f52\sqrt{m\sigma}\text{sign}(j)p_z^{(0)}=0$. 
For simplicity, we take $p^{(0)}=1$.
For the order $1$ term in $c(\tau,z,j)$, we  have \[\f5{2}\sqrt{m\sigma}\text{sign}(j)p_z^{(-1)} +ia_{31}\text{sign}(j)p^{(0)}=0.\]
To this end, we take $p^{(-1)}=-i\f2{5\sqrt{m\sigma}}\p_z^{-1}a_{31}$.
Note that we have chosen $a_{31}$ such that it does not have zero frequency component. 
Hence, $\p_z^{-1}a_{31}$ is well-defined.
For the order $\f12$ term in $c(\tau,z,j)$, we should have 
\[\f52\sqrt{m\sigma}\text{sign}(j)p_z^{(-2)} +\f{2}{5\sqrt{m\sigma}}a_{31}\text{sign}(j)\p_z^{-1}a_{31}+i(\gamma_{\tau}+a_{33})=0.\] 
The above equation has a solution  if and only if
\begin{align*}
\int_0^{2\pi}\left(\f2{5\sqrt{m\sigma}}a_{31}\text{sign}(j)\p_z^{-1}a_{31}+i(\p_\tau\gamma_0+a_{33})\right)dz=0.
\end{align*}
We choose 
\[
\gamma_0(\tau)=\f{i}{2\pi}\int_0^{\tau}\int_0^{2\pi}\left(\f{2}{5\sqrt{m\sigma}}a_{31}\text{sign}(j)\p_z^{-1}a_{31}+ia_{33}\right)dzds,\]
so that the above condition is satisfied. 
Then, we obtain the expression for $p^{(-2)}$:
\begin{align*}
p^{(-2)}=-\f2{5\sqrt{m\sigma}}\p^{-1}_z\left(\f{2}{5\sqrt{m\sigma}}a_{31}\text{sign}(j)\p_z^{-1}a_{31}+i(\gamma_{\tau}+a_{33})\right).
\end{align*}

Hence, we have constructed the operator $A^{(1)}$.
In particular, for each $\tau\in [0,T]$, 
\begin{equation} \label{PhaseEstimate}
 \|\gamma(\tau, \cdot)\|_{H^{s}} + \sup_{j \in \Z}\|(p-1)(\tau, \cdot, j) \|_{H^{s}} \lesssim \|(\eta,\psi)\|_{\H^s}
\end{equation}
Hence, using Lemma \ref{t:Einv}, the operator $A^{(1)}$ is invertible.
Moreover, 
\begin{equation*}
(A^{(1)})^{-1} \mL_7^{(1)}A^{(1)} = \mD_+ + R_{8}^+, 
\end{equation*}
where the remainder term $R_{8}^+$ satisfies 
\begin{equation} \label{REightEst}
 \|R_{8}^+ f \|_{H^s} \lesssim \|(\eta, \psi)\|_{\H^2} \|f\|_{H^s} + \|(\eta, \psi)\|_{\H^{s}} \|f\|_{H^2}.
\end{equation}
Similarly, one can also construct an invertible operator $A^{(2)}$ such that for 
\begin{equation*}
 \mD_- = \p_\tau - iT, \quad  \mL_7^{(2)} = \p_\tau -iT+a_{31}\p_z -ia_{33}|D_z|^\f12+ a_{32},   
\end{equation*}
one has the conjugation property 
\begin{equation*}
(A^{(2)})^{-1} \mL_7^{(2)}A^{(2)} = \mD_- + R_{8}^-, 
\end{equation*}
where the remainder term $R_{8}^-$ satisfies the same estimate as \eqref{REightEst}.
Setting the matrix operators
\begin{equation} \label{UDef}
\mathcal{U} =  \begin{pmatrix}
 A^{(1)}    & 0 \\ 
0 &  A^{(2)}
\end{pmatrix}, \quad \mathcal{U}^{-1} =  \begin{pmatrix}
  (A^{(1)})^{-1}    & 0 \\ 
0 &   (A^{(2)})^{-1}
\end{pmatrix}, \quad \mL_8 =  \begin{pmatrix}
  \mD_+    & 0 \\ 
0 &   \mD_-
\end{pmatrix},
\end{equation}
we then have the conjugation
\begin{equation} \label{mL78Conjugation}
 \mL_8 + \mR_8 = \mU^{-1}(\mL_7 +\mR_7)\mU, \quad \mR_8 := \mU^{-1}\mR_7\mU + (\mU^{-1}\mL_7 \mU- \mL_8).
\end{equation}
The remainder matrix $\mR_8$ satisfies for $s \geq s_0$,
\begin{equation} \label{mREightEst}
 \|\mR_8 \mathbf{f} \|_{C([0,T];H^{s}\times H^s)} \lesssim \|(\eta, \psi)\|^\theta_{C([0,T];\H^{s+\f52})} \|\mathbf{f}\|_{C([0,T];H^{s}\times H^s)}.
\end{equation}

\section{Well-posedness of the linear system} \label{s:Well}
In this section, we prove the well-posedness of the linear system 
\begin{equation} \label{CauchyZero}
    \mathbf{P}'(\mathbf{u})[\tilde{\mathbf{u}}]= \tilde{\mathbf{f}}, \quad \tilde{\mathbf{u}}|_{t=0} = (\tilde{\eta}_{in}, \tilde{\psi}_{in})^T,
\end{equation}
where the $\mathbf{P}'(\mathbf{u})[\tilde{\mathbf{u}}]$ is defined in \eqref{PuTildeu}.
We assume that
\begin{equation*}
\|(\eta, \psi)\|_{C([0,T]; \H^{s+\f52})} + \| \tilde{\mathbf{f}}\|_{C([0,T];\H^s)}<\delta, 
\end{equation*}
for some small constant $\delta$.
While it is possible to prove the well-posedness of this linear system directly, we will make use of the result in the reductions of $\mathbf{P}'(\mathbf{u})[\tilde{\mathbf{u}}]$ in previous two sections.

We start from the linear system 
\begin{equation} \label{CauchyEight}
  (\mL_8 + \mR_8)\mathbf{h} = \mathbf{f}_8, \quad \mathbf{h}|_{\tau=0} = \mathbf{h}_{in}. 
\end{equation}
By \eqref{mREightEst}, we get a bound for the remainder matrix operator $\mR_8$.
For $s \geq s_0$,
\begin{equation*}
 \|\mR_8 \mathbf{f} \|_{C([0,T];H^{s}\times H^s)} \lesssim \delta^\theta \|\mathbf{f}\|_{C([0,T];H^{s}\times H^s)}.
\end{equation*}
We prove the following well-posedness result of the system \eqref{CauchyEight}.
\begin{lemma} \label{t:CauchyEight}
Let $T>0$, and $s_1\geq 0$.
Suppose $\mathbf{h}_{in} \in H^{s_1}(\T)\times H^{s_1}(\T)$, $\mathbf{f}_8\in C([0,T];H^{s_1}(\T)\times H^{s_1}(\T))$.
There exists $\varepsilon_8>0$ depending on $T$ such that if 
\begin{equation*}
 \|\mR_8 \mathbf{f} \|_{C([0,T];H^{s_1}\times H^{s_1})} \lesssim \varepsilon_8 \|\mathbf{f}\|_{C([0,T];H^{s_1}\times H^{s_1})},
\end{equation*}
there exists a unique solution $\mathbf{h}\in C([0,T];H^{s_1}\times H^{s_1})$ of the Cauchy problem \eqref{CauchyEight} satisfying the estimate:
For $s\in [0, s_1]$, 
\begin{equation} \label{mR8EnergyEst}
 \| \mathbf{h}\|_{C([0,T]; H^s\times H^s)} \lesssim_s e^{CT}\left(\|\mathbf{h}_{in} \|_{H^s\times H^s} + \| \mathbf{f}_8\|_{C([0,T]; H^s\times H^s)}\right), 
\end{equation}
where the constant $C$ depends on $\varepsilon_8$.
\end{lemma}

\begin{proof}
We decompose $\mathbf{h} = \mathbf{h}_1+ \mathbf{h}_2$, where
\begin{equation} \label{mLEightCauchy}
\left\{
\begin{array}{lr}
\mL_8 \mathbf{h}_1 = \mathbf{f}_8 &  \\
\mathbf{h}_1|_{\tau=0} = \mathbf{h}_{in},&
\end{array}
\right. \quad
\left\{
\begin{array}{lr}
\mL_8 \mathbf{h}_2 = -\mR_8 \mathbf{h}_1 -\mR_8 \mathbf{h}_2 &  \\
\mathbf{h}_2|_{\tau=0} = \mathbf{0}.&
\end{array}
\right.
\end{equation}
The first system in \eqref{mLEightCauchy} is a linear system that can be solved explicitly.
There exists a unique solution $\mathbf{h}_1\in C([0,T]; H^s\times H^s)$, such that
\begin{equation}\label{es-h1}
 \| \mathbf{h}_1\|_{C([0,T]; H^s\times H^s)} \lesssim \|\mathbf{h}_{in} \|_{H^s\times H^s} + T\| \mathbf{f}_8\|_{C([0,T]; H^s\times H^s)}.
\end{equation}
We introduce the operator $\mathfrak{T}$ to denote the solution map of the first system in \eqref{mLEightCauchy}:
\begin{equation*}
\mathbf{h}_1:=\mathfrak{T}(\mathbf{f}_8,\mathbf{h}_{in}).
\end{equation*}
Thus, to solve the second system in \eqref{mLEightCauchy}, we need to prove that the following operator $\mathfrak{T}_1$ has a fixed point:
\begin{equation*}
\mathbf{w}= \mathfrak{T}_1(\mathbf{w}): =-\mathfrak{T}(\mathcal{R}_8\mathbf{h}_1+\mathcal{R}_8\mathbf{w},\mathbf{0}).
\end{equation*}
From \eqref{es-h1}, we have the uniform bound
\begin{align*}
&\|\mathfrak{T}_1(\mathbf{w})\|_{C([0,T]; H^s\times H^s)}\lesssim T\| \mathcal{R}_8(\mathbf{h}_1+\mathbf{w})\|_{C([0,T]; H^s\times H^s)}\\
\lesssim &T\|\mR_8\|_{C([0,T]; \mL (H^s\times H^s))}\left(\|\mathbf{h}_{in} \|_{H^s\times H^s} + T\| \mathbf{f}_8\|_{C([0,T]; H^s\times H^s)}+\| \mathbf{w}\|_{C([0,T]; H^s\times H^s)}\right), 
\end{align*}
and the contraction estimate
\begin{align*}
\|\mathfrak{T}_1(\mathbf{w}_1-\mathbf{w}_2)\|_{C([0,T]; H^s\times H^s)}&\lesssim T\| \mathcal{R}_8(\mathbf{w}_1-\mathbf{w}_2)\|_{C([0,T]; H^s\times H^s)}\\
&\lesssim T\|\mR_8\|_{C([0,T]; \mL (H^s\times H^s))}\| \mathbf{w}_1-\mathbf{w}_2\|_{C([0,T]; H^s\times H^s)}.
\end{align*}
From the above estimates,  if $\varepsilon_8$ is chosen such that $\varepsilon_8 T<\f12$, it follows that $\mathfrak{T}_1$ is a contraction. 
We then obtain the estimate from the contraction mapping principle:
\begin{align*}
 \| \mathbf{h}\|_{C([0,T]; H^s\times H^s)} \lesssim_s& (1+ T\|\mR_8\|_{C([0,T]; \mL (H^s\times H^s))})\|\mathbf{h}_{in} \|_{H^s\times H^s} + T\| \mathbf{f}_8\|_{C([0,T]; H^s\times H^s)} \\
+& T^2 \|\mR_8\|_{C([0,T]; \mL (H^s\times H^s))} \| \mathbf{f}_8\|_{C([0,T]; H^s\times H^s)}.
\end{align*}
This leads to the energy estimate \eqref{mR8EnergyEst}.
\end{proof}

Going back to the linear system 
\begin{equation} \label{CauchySeven}
  (\mL_7 + \mR_7) \mathbf{h} = \mathbf{f}_7, \quad \mathbf{h}|_{\tau=0} = \mathbf{h}_{in}.
\end{equation}
Using the conjugation \eqref{mL78Conjugation}, it can be written as
\begin{equation*} 
 (\mL_8 + \mR_8) (\mU^{-1}\mathbf{h}) = \mU^{-1}\mathbf{f}_7, \quad \mU^{-1}\mathbf{h}|_{\tau=0} = \mU^{-1}\mathbf{h}_{in}.  
\end{equation*}
We then use Lemma \ref{t:Einv} to estimate the inverse of $A^{(1)}$ and $A^{(2)}$.
We get by \eqref{PhaseEstimate} the following well-posedness result.
\begin{lemma} \label{t:CauchySeven}
Let $T>0$, and $s_1\geq 0$.
Suppose $\mathbf{h}_{in} \in H^{s_1}(\T)\times H^{s_1}(\T)$, $\mathbf{f}_7\in C([0,T];H^{s_1}(\T)\times H^{s_1}(\T))$.
There exists $\varepsilon_7>0$ depending on $T$ such that if 
\begin{equation*}
 \|\mR_7 \mathbf{f} \|_{C([0,T];H^{s_1}\times H^{s_1})} \lesssim \varepsilon_7 \|\mathbf{f}\|_{C([0,T];H^{s_1}\times H^{s_1})},
\end{equation*}
there exists a unique solution $\mathbf{h}\in C([0,T];H^{s_1}\times H^{s_1})$ of the Cauchy problem \eqref{CauchySeven} satisfying the estimate:
For $s\in [0, s_1]$,
\begin{equation} \label{mR7EnergyEst}
 \| \mathbf{h}\|_{C([0,T]; H^s\times H^s)} \lesssim_s e^{CT}\left(\|\mathbf{h}_{in} \|_{H^s\times H^s} + \| \mathbf{f}_7\|_{C([0,T]; H^s\times H^s)}\right), 
\end{equation}
where the constant $C$ depends on $\varepsilon_7$.
\end{lemma}

$\mU$ and $\mU^{-1}$ are unitary matrices with constant coefficients.
Hence, the system
\begin{equation*} 
  (\mL_6 + \mR_6)\mathbf{h} = \mathbf{f}_6, \quad \mathbf{h}|_{\tau=0} = \mathbf{h}_{in}. 
\end{equation*}
has the same existence result and energy estimate as the system \eqref{CauchySeven}.

For the linear system 
\begin{equation*} 
  (\mL_5 + \mR_5) \mathbf{h} = \mathbf{f}_5, \quad \mathbf{h}|_{\tau=0} = \mathbf{h}_{in},
\end{equation*}
using the conjugation \eqref{TLFiveConjugation}, and the bound 
\eqref{BBInverseDef}, it has the same existence result and energy estimate as the system \eqref{CauchySeven}.

Next, we consider the linear system 
\begin{equation} \label{CauchyFour}
  (\mL_4 + \mR_4) \mathbf{h} = \mathbf{f}_4, \quad \mathbf{h}|_{\tau=0} = \mathbf{h}_{in}.
\end{equation}
Using the conjugation \eqref{SLFourConjugation}, it can be written as
\begin{equation*} 
 (\mL_5 + \mR_5) (M^{-1}\mathbf{h}) = M^{-1}\mathbf{f}_5, \quad M^{-1}\mathbf{h}|_{\tau=0} = M^{-1}\mathbf{h}_{in}.  
\end{equation*}
In view of the bound \eqref{MpmBound}, we get the following well-posedness result.
\begin{lemma} \label{t:CauchyFour}
Let $T>0$, and $s_1\geq 0$.
Suppose $\mathbf{h}_{in} \in H^{s_1}(\T)\times H^{s_1}(\T)$, $\mathbf{f}_4\in C([0,T];H^{s_1}(\T)\times H^{s_1}(\T))$.
There exists $\varepsilon_4>0$ depending on $T$ such that if 
\begin{equation*}
 \|\mR_4 \mathbf{f} \|_{C([0,T];H^{s_1}\times H^{s_1})} \lesssim \varepsilon_4 \|\mathbf{f}\|_{C([0,T];H^{s_1}\times H^{s_1})},
\end{equation*}
there exists a unique solution $\mathbf{h}\in C([0,T];H^{s_1}\times H^{s_1})$ of the Cauchy problem \eqref{CauchyFour} satisfying the estimate:
For $s\in [0, s_1]$, 
\begin{equation} \label{mR4EnergyEst}
 \| \mathbf{h}\|_{C([0,T]; H^s\times H^s)} \lesssim_s e^{CT}\left(\|\mathbf{h}_{in} \|_{H^s\times H^s} + \| \mathbf{f}_4\|_{C([0,T]; H^s\times H^s)}\right),  
\end{equation}
where the constant $C$ depends on $\varepsilon_4$.
\end{lemma}

We then consider the linear system 
\begin{equation} \label{CauchyThree}
  \mL_3  \mathbf{h} = \mathbf{f}_3, \quad \mathbf{h}|_{\tau=0} = \mathbf{h}_{in}.
\end{equation}
Using the conjugation \eqref{SLThreeConjugation}, it can be written as
\begin{equation*} 
 (\mL_4 + \mR_4) (S^{-1}\mathbf{h}) = S^{-1}\mathbf{f}_4, \quad S^{-1}\mathbf{h}|_{\tau=0} = S^{-1}\mathbf{h}_{in}.  
\end{equation*}
We then get the well-posedness result for \eqref{CauchyThree}.

\begin{lemma} \label{t:CauchyThree}
Let $T>0$, and $s_1\geq s_0$.
Suppose $\mathbf{h}_{in} \in \H^{s_1}(\T)$, $\mathbf{f}_3\in C([0,T]; \H^{s_1}(\T))$.
There exists $\varepsilon_3>0$ depending on $T$ such that if 
\begin{equation*}
 \|(\eta, \psi)\|_{C([0,T]; \H^{s_1+\f52})} \lesssim \varepsilon_3 ,
\end{equation*}
there exists a unique solution $\mathbf{h}\in C([0,T]; \H^{s_1})$ of the Cauchy problem \eqref{CauchyThree} satisfying the estimate:
For $s\in [0, s_1]$, 
\begin{equation} \label{mR3EnergyEst}
 \| \mathbf{h}\|_{C([0,T]; \H^s)} \lesssim_s e^{CT}\left(\|\mathbf{h}_{in} \|_{\H^s} + \| \mathbf{f}_3\|_{C([0,T]; \H^s)}\right), 
\end{equation}
where the constant $C$ depends on $\varepsilon_3$.
\end{lemma}

For the linear system 
\begin{equation} \label{CauchyTwo}
  \mL_2  \mathbf{h} = \mathbf{f}_2, \quad \mathbf{h}|_{\tau=0} = \mathbf{h}_{in},
\end{equation}
applying \eqref{LThreeMatrix}, it can be written as
\begin{equation*} 
\mL_3 (Q^{-1}\mathbf{h}) = P^{-1}\mathbf{f}_2, \quad Q^{-1}\mathbf{h}|_{\tau=0} = Q^{-1}\mathbf{h}_{in}.
 \end{equation*}
Using the bound \eqref{L2inv} and Lemma \ref{t:CauchyThree}, we get the following well-posedness result.
\begin{lemma} \label{t:CauchyTwo}
Let $T>0$, and $s_1\geq s_0$.
Suppose $\mathbf{h}_{in} \in \H^{s_1}(\T)$, $\mathbf{f}_2\in C([0,T]; \H^{s_1}(\T))$.
There exists $\varepsilon_2>0$ depending on $T$ such that if 
\begin{equation*}
 \|(\eta, \psi)\|_{C([0,T]; \H^{s_1+\f52})} \lesssim \varepsilon_2 ,
\end{equation*}
there exists a unique solution $\mathbf{h}\in C([0,T]; \H^{s_1})$ of the Cauchy problem \eqref{CauchyTwo} satisfying the estimate:
For $s\in [0, s_1]$, 
\begin{equation} \label{mR2EnergyEst}
 \| \mathbf{h}\|_{C([0,T]; \H^s)} \lesssim_s e^{CT}\left(\|\mathbf{h}_{in} \|_{\H^s} + \| \mathbf{f}_2\|_{C([0,T]; \H^s)}\right), 
\end{equation}
where the constant $C$ depends on $\varepsilon_2$.
\end{lemma}

For the linear system 
\begin{equation} \label{CauchyOne}
  \mL_1  \mathbf{h} = \mathbf{f}_1, \quad \mathbf{h}|_{t=0} = \mathbf{h}_{in},
\end{equation}
by \eqref{LTwoMatrix}, it can be written as 
\begin{equation*} 
\mL_2(\mathcal{A}^{-1}\mathbf{h}) = \mathcal{A}^{-1}\mathbf{f}_1, \quad \mathcal{A}^{-1}\mathbf{h}|_{\tau=0} = \mathcal{A}^{-1}\mathbf{h}_{in}.
 \end{equation*}
Using the bound \eqref{AAInverseBound} and Lemma \ref{t:CauchyTwo}, we get the following well-posedness result.
\begin{lemma} \label{t:CauchyOne}
Let $T>0$, and $s_1\geq s_0$.
Suppose $\mathbf{h}_{in} \in \H^{s_1}(\T)$, $\mathbf{f}_1\in C([0,T]; \H^{s_1}(\T))$.
There exists $\varepsilon_1>0$ depending on $T$ such that if 
\begin{equation*}
 \|(\eta, \psi)\|_{C([0,T]; \H^{s_1+\f52})} \lesssim \varepsilon_1 ,
\end{equation*}
there exists a unique solution $\mathbf{h}\in C([0,T]; \H^{s_1})$ of the Cauchy problem \eqref{CauchyOne} satisfying the estimate:
For $s\in [0, s_1]$, 
\begin{equation} \label{mR1EnergyEst}
 \| \mathbf{h}\|_{C([0,T]; \H^s)} \lesssim_s e^{CT}\left(\|\mathbf{h}_{in} \|_{\H^s} + \| \mathbf{f}_1\|_{C([0,T]; \H^s)}\right),
\end{equation}
where the constant $C$ depends on $\varepsilon_1$.
\end{lemma}

For the linear system 
\begin{equation} \label{CauchyZero1}
  \mL_0  \mathbf{h} = \mathbf{f}_0, \quad \mathbf{h}|_{t=0} = \mathbf{h}_{in},
\end{equation}
in view of the conjugation \eqref{LOneMatrix}, it can be written as
\begin{equation*} 
\mL_1(\mathcal{B}^{-1}\mathbf{h}) = \mathcal{B}^{-1}\mathbf{f}_0, \quad \mathcal{B}^{-1}\mathbf{h}|_{t=0} = \mathcal{B}^{-1}\mathbf{h}_{in}.
 \end{equation*}
Using the bound \eqref{BBInverseBound} and Lemma \ref{t:CauchyOne}, we get the following well-posedness result.
\begin{lemma} \label{t:CauchyZero1}
Let $T>0$, and $s_1\geq s_0$.
Suppose $\mathbf{h}_{in} \in \H^{s_1}(\T)$, $\mathbf{f}_0\in C([0,T]; \H^{s_1}(\T))$.
There exists $\varepsilon_0>0$ depending on $T$ such that if 
\begin{equation*}
 \|(\eta, \psi)\|_{C([0,T]; \H^{s_1+\f52})} \lesssim \varepsilon_0,
\end{equation*}
there exists a unique solution $\mathbf{h}\in C([0,T]; \H^{s_1})$ of the Cauchy problem \eqref{CauchyOne} satisfying the estimate:
For $s\in [0, s_1]$, 
\begin{equation} \label{mR01EnergyEst}
 \| \mathbf{h}\|_{C([0,T]; \H^s)} \lesssim_s e^{CT}\left(\|\mathbf{h}_{in} \|_{\H^s} + \| \mathbf{f}_0\|_{C([0,T]; \H^s)}\right),
\end{equation}
where the constant $C$ depends on $\varepsilon_0$.
\end{lemma}

Finally, for the system \eqref{CauchyZero}, using the conjugation \eqref{ZConjugation}, it can be written as
\begin{equation*}
\mathcal{L}_0\mathcal{Z}^{-1}\tilde{\mathbf{u}}=\mathcal{Z}^{-1}\tilde{\mathbf{f}},\quad \mathcal{Z}^{-1}\tilde{\mathbf{u}}|_{t=0} = \mathcal{Z}^{-1}\tilde{\mathbf{u}}_{in}.
\end{equation*}
Using the bound \eqref{ZBoundInverse} and Lemma \ref{t:CauchyZero1}, we get the following well-posedness result.
\begin{lemma} \label{t:CauchyZero}
Let $T>0$, and $s_1\geq s_0$.
Suppose $\tilde{\mathbf{u}}_{in} \in \H^{s_1}(\T)$, $\tilde{\mathbf{f}}\in C([0,T]; \H^{s_1}(\T))$.
There exists $\delta>0$ depending on $T$ such that if 
\begin{equation*}
 \|(\eta, \psi)\|_{C([0,T]; \H^{s_1+\f52})}<\delta,
\end{equation*}
there exists a unique solution $\tilde{\mathbf{u}}\in C([0,T]; \H^{s_1})$ of the Cauchy problem \eqref{CauchyZero} satisfying the estimate:
For $s\in [0, s_1]$, 
\begin{equation} \label{TildeU0EnergyEst}
 \|\tilde{\mathbf{u}}\|_{C([0,T]; \H^s)} \lesssim_s e^{CT}\left(\|\tilde{\mathbf{u}}_{in} \|_{\H^s} + \|\tilde{\mathbf{f}}\|_{C([0,T]; \H^s)}\right),  
\end{equation}
where the constant $C$ depends on $\varepsilon_0$.
\end{lemma}

For an operator $\mM$, we write its adjoint operator with respect to the space-time scalar product $\langle \cdot, \cdot\rangle_{(t,x)} = \int_0^T \langle \cdot, \cdot\rangle_{\H^0}dt$ by  $\mM^*$.
For instance, we write
\begin{equation*}
 \mathcal{L}^*_0 : = \begin{pmatrix}
   -\p_t -V\p_x  &   -|D|+ \mathcal{R}^*_G \\
   a + \sigma \sum _{\al = 1}^4 (-1)^{5-\al}\p_x^{5-\al}\left(E_{5-\al}(\eta) \cdot\right)&  -\p_t - V\p_x - V_x 
    \end{pmatrix}
\end{equation*}
 for the adjoint operator of  $\mathcal{L}_0$, where $\mathcal{R}_G^*$ is the adjoint operator of $\mathcal{R}_G$. 
The adjoint operators $[\mathbf{P}'(\mathbf{u})]^*$, $\mathcal{L}_i^*$ ($i=1,\dots,8$), and $\mR_i^*$ ($i= 4, \dots, 8$) can be defined similarly.
One can check that Lemma \ref{t:CauchyEight}-Lemma \ref{t:CauchyZero1} also hold when the operators $\mathcal{L}_i$ ($i=0,\dots,3$), $\mathcal{L}_i + \mR_i$ ($i = 4, \cdots, 8$) and $\mathbf{P}'(\mathbf{u})$ are replaced by $\mathcal{L}_i^*$ ($i=0,\dots,3$), $\mathcal{L}_i^* + \mR_i^*$ ($i = 4,\cdots, 8$), and $[\mathbf{P}'(\mathbf{u})]^*$.
Furthermore, the same estimates hold for the backward adjoint Cauchy problems, in which adjoint linearized systems start from $t=T$ and propagate backward in time.

\section{Observability of the linearized system} \label{s:Observe}
In this Section, we prove the observability of the linearized system \eqref{PuTildeu}.
Again, we will exploit the reduction results in Section \ref{s:Top} and Section \ref{s:Lower}, and obtain the observability of the linearized system in a backward way.

First, we consider the observability for $\mL_8 + \mR_8$, and prove the following result.
\begin{lemma} \label{t:ObserveEight}
Let $T>0$,  and $\omega \subset \T$ be an open set.
Let $\textbf{h}_T \in H^{\f32} \times H^{\f32}(\T)$, and $\textbf{h} \in C([0,T]; H^{\f32}\times H^{\f32}(\T))$ be the solution of the backward Cauchy problem 
\begin{equation} \label{BackmLCauchy8}
 (\mL_8 + \mR_8)\mathbf{h} = \mathbf{0}, \quad \mathbf{h}|_{\tau=T} = \mathbf{h}_{T}.  
\end{equation}
There exists a constant $C_8 = C_8(T,\omega)$, such that if $\| \mR_8\|_{C([0,T]; \mL(H^{\f32}\times H^{\f32}))}< r_8$, for some constant $r_8$ small enough, one has the inequality
\begin{equation}\label{ObserveEight}
    \int_0^T \int_\omega |\Lambda\mathbf{h}(\tau,z)|^2 dzd\tau \geq C_8 \| \Lambda\mathbf{h}_{T}\|_{L^2\times L^2}^2,
\end{equation}
where $\Lambda$ is the Fourier multiplier whose symbol is defined in \eqref{LambdaSymbolDef}.
\end{lemma}

\begin{proof}
We decompose $\mathbf{h} = \mathbf{h}_1+ \mathbf{h}_2$, where
\begin{equation} \label{mLBackwardEightCauchy}
\left\{
\begin{array}{lr}
\mL_8 \mathbf{h}_1 = \mathbf{0} &  \\
\mathbf{h}_1|_{\tau=T} = \mathbf{h}_T,&
\end{array}
\right. \quad
\left\{
\begin{array}{lr}
(\mL_8 + \mR_8)\mathbf{h}_2 = -\mR_8 \mathbf{h}_1 &  \\
\mathbf{h}_2|_{\tau=T} = \mathbf{0}.&
\end{array}
\right.
\end{equation}
We first prove the $L^2\times L^2$ observability result, namely the inequality without $\Lambda$ in \eqref{ObserveEight}. 

For the first system in \eqref{mLBackwardEightCauchy}, since this is a vector equation, it suffices to consider its first component, namely, the scalar equation
\begin{equation} \label{mD+ScalarEqn}
\mD_+ h_1 = 0, \quad 
h_1|_{\tau=T} = h_{T,1}.
\end{equation}
Let $h_{T,1} = \sum_{n\in \Z} w_n e^{inz}$, so that $\| h_{T,1}\|_{L^2}^2 = \sum_{n\in \Z} |w_n|^2$.
We can then write
\begin{equation*}
    h_1(\tau,z) =  \sum_{n\in \Z} w_n e^{inz}e^{-i\mathfrak{t}(n)(\tau-T)} = \sum_{n\in \N} u_n(z) e^{-i\mathfrak{t}(n)\tau},
\end{equation*}
where $u_n(z)$ are given by
\begin{equation*}
 u_n(z) = \left\{
\begin{array}{lr}
e^{i\mathfrak{t}(n)T}(w_n e^{inz}+ w_{-n}e^{-inz}), \quad \text{for } n\geq1 &  \\
w_0, \quad \text{for } n = 0.&
\end{array}
\right. 
\end{equation*}
By Ingham inequality Lemma \ref{t:Ingham}, we have 
\begin{equation}\label{h1geqzn}
\int_0^T \int_\omega |h_1(\tau,z)|^2 dzd\tau \geq C(T) \sum_{n\in\mathbb{N}}\int_{\omega_0}|u_n(z)|^2dz,
\end{equation}
where $\omega_0=(a, b)\subset\omega$ is chosen such that, for $n\geq1$,
\begin{equation*}
\left|\f{\sin(n(b-a))}{n}\right|\leq \sin(b-a). 
\end{equation*}
When $n=0$, 
\begin{equation*}
\int_{\omega_0}|u_0(z)|^2dz=|\omega_0|^2(b-a).
\end{equation*}
When $n\geq1$, we compute
\begin{align*}
\int_{\omega_0}|u_n(z)|^2dz=&\int_{\omega_0}\left(|w_n|^2+|w_{-n}|^2+w_n\bar{w}_ne^{2inz}+\bar{w}_n\bar{w}_ne^{-2inz}\right)dz\\
\geq&(b-a)(|w_n|^2+|w_{-n}|^2)-|w_n||w_{-n}|\left(\left|\int_{\omega_0}e^{2inz}dz\right|+\left|\int_{\omega_0}e^{-2inz}dz\right|\right)\\
=&(b-a)(|w_n|^2+|w_{-n}|^2)-2|w_n||w_{-n}|\left|\f{\sin(n(b-a))}{n}\right|\\
\geq&\left\{b-a-\left|\f{\sin(n(b-a))}{n}\right|\right\}(|w_n|^2+|w_{-n}|^2).
\end{align*}
Thus, we obtain that
\begin{equation*}
\int_{\omega_0}|u_n(z)|^2dz\geq \left(b-a-\sin(b-a)\right)(|w_n|^2+|w_{-n}|^2).
\end{equation*}
Taking the sum of the above inequality for $n\in \N$, we get
\begin{equation*}
\sum_{n\in\mathbb{N}}\int_{\omega_0}|u_n(z)|^2dz\geq C(\omega_0)\sum_{n\in\mathbb{Z}}|w_n|^2.
\end{equation*}
Combining this estimate with \eqref{h1geqzn}, we obtain the estimate
\begin{equation*}
\int_0^T \int_\omega |h_1(\tau,z)|^2 dzd\tau \geq C(T)C(\omega_0)\|h_{T,1}\|_{L^2}^2.
\end{equation*}
For the second component of the first system of system \eqref{mLBackwardEightCauchy}, we have the same estimate, so that putting the inequalities for two components together, 
\begin{equation*}
\int_0^T \int_\omega |\mathbf{h}_1(\tau,z)|^2 dzd\tau \geq C(T)C(\omega_0)\|\mathbf{h}_{T}\|_{L^2\times L^2}^2.
\end{equation*}

For the second system in \eqref{mLBackwardEightCauchy}, we use the energy estimate \eqref{mR8EnergyEst} for the Cauchy problem \eqref{CauchyEight} for two times, 
\begin{equation*} 
 \| \mathbf{h}_2\|_{C([0,T]; L^2\times L^2)} \leq Ce^{CT}r_8 \| \mathbf{h}_1\|_{C([0,T]; L^2\times L^2)}\leq C^2e^{2CT}r_8\|\mathbf{h}_T\|_{L^2\times L^2}.   
\end{equation*}
Thus, if $r_8$ is chosen small enough, so that $2C^2e^{2CT}\sqrt{T}r_8<C(T)C(\omega_0)$, then using the inequality $(a+b)^2 \geq \f12 a^2 - b^2$,
\begin{align*}
&\int_0^T \int_\omega |\mathbf{h}(\tau,z)|^2 dxdt \geq \f12 \int_0^T \int_\omega |\mathbf{h}_1(\tau,z)|^2 dzd\tau-\int_0^T \int_\omega |\mathbf{h}_2(\tau,z)|^2 dzd\tau \\
\geq& \f12
C(T)C(\omega_0) \| \mathbf{h}_{T}\|_{L^2\times L^2}^2-C^4e^{4CT}Tr^2_8\| \mathbf{h}_{T}\|^2_{L^2\times L^2} \geq \f14C(T)C(\omega_0) \| \mathbf{h}_{T}\|_{L^2\times L^2}^2.
\end{align*}
Hence, we obtain the $L^2\times L^2$ observability result.

Now, we apply $\Lambda$ to systems in \eqref{mLBackwardEightCauchy}.
Since $\Lambda$ commutes with $\p_\tau$, we get 
\begin{equation} \label{mLBackwardEightCauchy32}
\left\{
\begin{array}{lr}
\mL_8 \Lambda\mathbf{h}_1 = \mathbf{0} &  \\
\Lambda\mathbf{h}_1|_{\tau=T} = \Lambda\mathbf{h}_T,&
\end{array}
\right. \quad
\left\{
\begin{array}{lr}
(\mL_8 + \Lambda\mR_8\Lambda^{-1})\Lambda\mathbf{h}_2 = -\Lambda\mR_8 \mathbf{h}_1 &  \\
\Lambda\mathbf{h}_2|_{\tau=T} = \mathbf{0}.&
\end{array}
\right.
\end{equation}
Applying the same analysis above to \eqref{mLBackwardEightCauchy32}
and taking $C_8=\f14C(T)C(\omega_0)$, we obtain the inequality \eqref{ObserveEight}. 
\end{proof}

We then consider the backward linear system 
\begin{equation} \label{CauchybackSeven}
  (\mL_7 + \mR_7) \mathbf{h} = \mathbf{0}, \quad \mathbf{h}|_{\tau=T} = \mathbf{h}_T.
\end{equation}
Recall that we have the conjugation
\begin{equation*} 
 (\mL_8 + \mR_8) (\mU^{-1}\mathbf{h}) = \mathbf{0}, \quad \mU^{-1}\mathbf{h}|_{\tau=T} = \mU^{-1}\mathbf{h}_{T}.  
\end{equation*}
Using Lemma \ref{t:Einv} for inverse operators of $A^{(1)}$ and $A^{(2)}$.
We get by \eqref{PhaseEstimate} the following observability result.

\begin{lemma} \label{t:ObserveSeven}
Let $T>0$, and $\omega \subset \T$ be an open set.
Let $\textbf{h}_T \in H^{\f32} \times H^{\f32}(\T)$, and $\textbf{h} \in C([0,T]; H^{\f32}\times H^{\f32}(\T))$ be the solution of the backward Cauchy problem \eqref{CauchybackSeven}.
There exists a constant $C_7 = C_7(T,\omega)$, such that if $\| \mR_7\|_{C([0,T]; \mL(H^{\f32}\times H^{\f32}))}< r_7$, for some constant $r_7$ small enough, one has the inequality
\begin{equation}\label{ObserveSeven}
    \int_0^T \int_\omega |\Lambda\mathbf{h}(\tau,z)|^2 dzd\tau \geq C_7 \| \Lambda\mathbf{h}_{T}\|_{L^2\times L^2}^2.
\end{equation}
\end{lemma}

Since $\mU$ and $\mU^{-1}$ are unitary matrices with constant coefficients, the system
\begin{equation*} 
  (\mL_6 + \mR_6)\mathbf{h} = \mathbf{0}, \quad \mathbf{h}|_{\tau=T} = \mathbf{h}_{T}. 
\end{equation*}
has the same observability result as \eqref{CauchybackSeven}.

For the backward linear system 
\begin{equation*} 
  (\mL_5 + \mR_5) \mathbf{h} = \mathbf{0}, \quad \mathbf{h}|_{\tau=T} = \mathbf{h}_{T},
\end{equation*}
applying the conjugation \eqref{TLFiveConjugation}, and the bound 
\eqref{BBInverseDef}, it has the same observability result as the backward linear system \eqref{CauchybackSeven}.

Next, we consider the backward linear system 
\begin{equation} \label{CauchybackFour}
  (\mL_4 + \mR_4) \mathbf{h} = \mathbf{0}, \quad \mathbf{h}|_{\tau=T} = \mathbf{h}_{T}.
\end{equation}
Using the conjugation \eqref{SLFourConjugation}, it can be written as
\begin{equation*} 
 (\mL_5 + \mR_5) (M^{-1}\mathbf{h}) = \mathbf{0}, \quad M^{-1}\mathbf{h}|_{\tau=T} = M^{-1}\mathbf{h}_{T}.  
\end{equation*}
With the bound \eqref{MpmBound}, we get the following observability result.
\begin{lemma} \label{t:ObserveFour}
Let $T>0$, and $\omega \subset \T$ be an open set.
Let $\textbf{h}_T \in H^{\f32} \times H^{\f32}(\T)$, and $\textbf{h} \in C([0,T]; H^{\f32}\times H^{\f32}(\T))$ be the solution of the backward Cauchy problem \eqref{CauchybackFour}.
There exists a constant $C_4 = C_4(T,\omega)$, such that if $\| \mR_4\|_{C([0,T]; \mL(H^\f32\times H^\f32))}< r_4$, for some constant $r_4$ small enough, one has the inequality
\begin{equation}\label{ObserveFour}
    \int_0^T \int_\omega |\Lambda\mathbf{h}(\tau,y)|^2 dyd\tau \geq C_4 \| \Lambda\mathbf{h}_{T}\|_{L^2\times L^2}^2.
\end{equation}
\end{lemma}

As for the linear system 
\begin{equation} \label{CauchyBackThree}
  \mL_3  \mathbf{h} = \mathbf{0}, \quad \mathbf{h}|_{\tau=T} = \mathbf{h}_{T},
\end{equation}
by the conjugation \eqref{SLThreeConjugation}, it can be written as
\begin{equation*} 
 (\mL_4 + \mR_4) (S^{-1}\mathbf{h}) = \mathbf{0}, \quad S^{-1}\mathbf{h}|_{\tau=T} = S^{-1}\mathbf{h}_{T}.  
\end{equation*}
We then get the observability result for \eqref{CauchyBackThree}.

\begin{lemma} \label{t:ObserveThree}
Let $T>0$, and $\omega \subset \T$ be an open set.
Let $\textbf{h}_T \in \H^0(\T)$, and $\textbf{h} \in C([0,T]; \H^0(\T))$ be the solution of the backward Cauchy problem \eqref{CauchyBackThree}.
There exists a constant $C_3 = C_3(T,\omega)$, such that if $\| (\eta, \psi)\|_{C([0,T]; \H^{s_0 +\f52})}< r_3$, for some constant $r_3$ small enough, one has the inequality
\begin{equation*}
\int_0^T \int_\omega |\Lambda S^{-1}\mathbf{h}(\tau, y)|^2 dyd\tau \geq C_4 \| \Lambda S^{-1}\mathbf{h}_{T}\|_{L^2\times L^2}^2,
\end{equation*}
which is equivalent to  
\begin{equation}\label{ObserveThree}
    \int_0^T \|\mathbf{h}(\tau,y)\|_{\H^0_y(\omega)}^2 d\tau \geq C_3 \| \mathbf{h}_{T}\|_{\H^0}^2.
\end{equation}
\end{lemma}

For the backward linear system 
\begin{equation} \label{CauchyBackTwo}
  \mL_2  \mathbf{h} = \mathbf{0}, \quad \mathbf{h}|_{t=T} = \mathbf{h}_{T},
\end{equation}
in view of \eqref{LThreeMatrix}, it can be written as
\begin{equation*} 
\mL_3 (Q^{-1}\mathbf{h}) = \mathbf{0}, \quad Q^{-1}\mathbf{h}|_{\tau=T} = Q^{-1}\mathbf{h}_{T}.
 \end{equation*}
Using the bound \eqref{AAInverseBound} and Lemma \ref{t:CauchyTwo}, we get the following observability result.
\begin{lemma} \label{t:ObserveTwo}
Let $T>0$, and $\omega \subset \T$ be an open set.
Let $\textbf{h}_T \in \H^0(\T)$, and $\textbf{h} \in C([0,T]; \H^0(\T))$ be the solution of the backward Cauchy problem \eqref{CauchyBackTwo}.
There exists a constant $C_2 = C_2(T,\omega)$, such that if $\| (\eta, \psi)\|_{C([0,T]; \H^{s_0 +\f52})}< r_2$, for some constant $r_2$ small enough, one has the inequality  
\begin{equation}\label{ObserveTwo}
    \int_0^T \|\mathbf{h}(t,y)\|_{\H^0_y(\omega)}^2 dt \geq C_2\| \mathbf{h}_{T}\|_{\H^0}^2.
\end{equation}
\end{lemma}

For the backward linear systems $\mL_1 \mathbf{h} = \mathbf{0}$, with $\mathbf{h}|_{t=T} = \mathbf{h}_{T}$, and
\begin{equation} \label{CauchyBackZero1}
  \mL_0 \mathbf{h} = \mathbf{0}, \quad \mathbf{h}|_{t=T} = \mathbf{h}_{T},
\end{equation}
similar to the argument as in the proof of Lemma \ref{t:CauchyOne} and Lemma \ref{t:CauchyZero1}, we can get the observability result for \eqref{CauchyBackZero1}. 
\begin{lemma} \label{t:ObserveZero1}
Let $T>0$, and $\omega \subset \T$ be an open set.
Let $\textbf{h}_T \in \H^0(\T)$, and $\textbf{h} \in C([0,T]; \H^0(\T))$ be the solution of the backward Cauchy problem \eqref{CauchyBackZero1}.
There exists a constant $C_0 = C_0(T,\omega)$, such that if $\| (\eta, \psi)\|_{C([0,T]; \H^{s_0 +\f52})}< r_0$, for some constant $r_0$ small enough, one has the inequality  
\begin{equation}\label{ObserveZero1}
    \int_0^T \|\mathbf{h}(t,x)\|_{\H^0_x(\omega)}^2 dt \geq C_0\| \mathbf{h}_{T}\|_{\H^0}^2.
\end{equation}
\end{lemma}

Finally, for the backward linear system 
\begin{equation} \label{CauchyBackZero}
    \mathbf{P}'(\mathbf{u})[\tilde{\mathbf{u}}]= \mathbf{0}, \quad \tilde{\mathbf{u}}|_{t=T} = \tilde{\mathbf{u}}_T =  (\tilde{\eta}_{end}, \tilde{\psi}_{end})^T,
\end{equation}
in view of the conjugation \eqref{ZConjugation}, it can be written as
\begin{equation*}
\mathcal{L}_0\mathcal{Z}^{-1}\tilde{\mathbf{u}}=\mathbf{0},\quad \mathcal{Z}^{-1}\tilde{\mathbf{u}}|_{t=T} = \mathcal{Z}^{-1}\tilde{\mathbf{u}}_{T}.
\end{equation*}
Using the bound \eqref{ZBoundInverse} and Lemma \ref{t:CauchyZero}, we get the following observability result for the linearized hydroelastic waves.
\begin{lemma} \label{t:ObserveZero}
Let $T>0$, and $\omega \subset \T$ be an open set.
Let $\tilde{\mathbf{u}}_T \in \H^0(\T)$, and $\tilde{\mathbf{u}} \in C([0,T]; \H^0(\T))$ be the solution of the backward Cauchy problem \eqref{CauchyBackZero}.
There exists a constant $C= C(T,\omega)$, such that if $\| (\eta, \psi)\|_{C([0,T]; \H^{s_0 +\f52})}< \delta$, for some constant $\delta$ small enough, one has the inequality  
\begin{equation}\label{ObserveZero}
    \int_0^T \|\tilde{\mathbf{u}}(t,x)\|_{\H^0_x(\omega)}^2 dt \geq C\| \tilde{\mathbf{u}}_{T}\|_{\H^0}^2.
\end{equation}
\end{lemma}

\section{Controllability of the linearized system} \label{s:Control}
In this Section, we prove the  $\H^s$ controllability of the linearized system \eqref{PuTildeu}.
Before going to this result, we first prove a $\H^0$ controllability result for \eqref{PuTildeu}.

\begin{lemma}\label{t:ControlH0}
Let $T>0$ and $\omega \subset \T$ be an open set. 
Let $[\mathbf{P}'(\mathbf{u})]^*$ be the adjoint operator of $\mathbf{P}'(\mathbf{u})$.
There exists $\delta_0\in (0,1)$ small enough such that, if $\| (\eta, \psi)\|_{C([0,T]; \H^{s_0 +\f52})}< \delta_0$, then for any real-valued $\mathbf{h}_{in}, \mathbf{h}_{end} \in \H^0(\T),$ and $\mathbf{q}\in C([0,T]; \H^0(\T))$, there exists a unique real-valued function $\mathbf{f}\in C([0,T]; \H^0(\T))$ that solves $[\mathbf{P}'(\mathbf{u})]^*\mathbf{f} = \mathbf{0}$ such that the only solution $\mathbf{h}\in C([0,T]; \H^0(\T))$ of the Cauchy problem
\begin{equation} \label{H0ControlEqn}
\mathbf{P}'(\mathbf{u}) \mathbf{h} =\text{diag}\{0, \chi_\omega \}\mathbf{f}+ \mathbf{q}, \quad 
\mathbf{h}|_{t=0} = \mathbf{h}_{in}
\end{equation}
satisfies $\mathbf{h}|_{t=T} = \mathbf{h}_{end}$.
Furthermore, the function $\mathbf{f}$ satisfies the control estimate
\begin{equation*}
 \|\mathbf{f}\|_{C([0,T]; \H^0)} \lesssim \|\mathbf{h}_{in}\|_{\H^0} + \|\mathbf{h}_{end}\|_{\H^0} + \|\mathbf{q}\|_{C([0,T]; \H^0)}.
\end{equation*}
\end{lemma}
We remark that the control function that leads to the linearized control problem is not unique.
The function $\mathbf{f}$ that we construct is unique in the sense that it is the unique solution to the first system in \eqref{mLBackwardAdjoint} below.

\begin{proof}
We first prove the existence of the real-valued function $\mathbf{f}\in C([0,T]; \H^0(\T))$.
For any $\mathbf{f}_1, \mathbf{g}_1 \in \H^0(\T)$, we consider real-valued $\mathbf{f}, \mathbf{g} \in C([0,T]; \H^0(\T))$ as the unique solutions of the backward adjoint Cauchy problems
\begin{equation}
\label{mLBackwardAdjoint}
\left\{
\begin{array}{lr}
[\mathbf{P}'(\mathbf{u})]^* \mathbf{f} = \mathbf{0} &  \\
\mathbf{f}|_{t=T} = \mathbf{f}_1,&
\end{array}
\right. \quad
\left\{
\begin{array}{lr}
[\mathbf{P}'(\mathbf{u})]^* \mathbf{g} = \mathbf{0} &  \\
\mathbf{g}|_{t=T} = \mathbf{g}_1,&
\end{array}
\right.
\end{equation}
and we define the bilinear form
\begin{equation*}
    \mathfrak{B}( \mathbf{f}_1,  \mathbf{g}_1) : = \int_0^T \langle \text{diag}\{0, \chi_\omega \}  \mathbf{f},  \mathbf{g}\rangle_{\H^0_x} \,dt,
\end{equation*}
and the linear form
\begin{equation*}
\mathbf{\Lambda}(\mathbf{g}_1) := \langle   \mathbf{h}_{end},  \mathbf{g}_1\rangle_{\H^0_x} - \langle   \mathbf{h}_{in},  \mathbf{g}(0,\cdot)\rangle_{\H^0_x} -\int_0^T \langle   \mathbf{q}(t,\cdot),  \mathbf{g}(t,\cdot)\rangle_{\H^0_x} \,dt,
\end{equation*}
where the scalar product $\langle   \cdot,  \cdot\rangle_{\H^0_x}$ is defined as
\begin{equation*}
\langle  (u_1, u_2),  (v_1, v_2)\rangle_{\H^0_x} : = \int_\T |D|^\f32 u_1(x) |D|^\f32\bar{v}_1(x) + u_2(x)\bar{v}_2(x) \,dx.
\end{equation*}
Using the energy estimate \eqref{mR01EnergyEst} for the backward adjoint problem, we get
\begin{equation*}
|\mathfrak{B}( \mathbf{f}_1,  \mathbf{g}_1)| \lesssim \|\mathbf{f}_1 \|_{\H^0} \|\mathbf{g}_1 \|_{\H^0}, \quad  |\mathbf{\Lambda}(\mathbf{g}_1)| \lesssim \left(\|\mathbf{h}_{in}\|_{\H^0} + \|\mathbf{h}_{end}\|_{\H^0} + \|\mathbf{q}\|_{C([0,T];\H^0)}\right) \| \mathbf{g}_1\|_{\H^0}.
\end{equation*}
In addition, the observability inequality \eqref{ObserveZero} implies that the bilinear form $\mathfrak{B}$ is coercive.
Therefore, according to Lax-Milgram theorem, there exists a unique real-valued $\mathbf{f}_1 \in \H^0(\T)$ such that
\begin{equation} \label{LaxMilgram}
\mathfrak{B}( \mathbf{f}_1,  \mathbf{g}_1) = \mathbf{\Lambda}(\mathbf{g}_1), \quad \forall \mathbf{g}_1 \in \H^0(\T).
\end{equation}
Moreover, $\mathbf{f}_1$ satisfies the estimate
\begin{equation*}
    \|\mathbf{f}_1\|_{\H^0} \lesssim \|\mathbf{h}_{in}\|_{\H^0} + \|\mathbf{h}_{end}\|_{\H^0} + \|\mathbf{q}\|_{C([0,T];\H^0)}.
\end{equation*}
For the function $\mathbf{f}\in C([0,T]; \H^0(\T))$ that is the solution of the first system in \eqref{mLBackwardAdjoint}, we let $\mathbf{h}$ be the solution of the Cauchy problem \eqref{H0ControlEqn}, whose existence is given in Lemma \ref{t:ObserveZero}.
Using the definition,
\begin{align*}
 &0 = \mathfrak{B}( \mathbf{f}_1,  \mathbf{g}_1) - \mathbf{\Lambda}(\mathbf{g}_1) \\
 =& \int_0^T \langle \text{diag}\{0, \chi_\omega \}  \mathbf{f},  \mathbf{g}\rangle_{\H^0} \,dt - \langle   \mathbf{h}_{end},  \mathbf{g}_1\rangle_{\H^0} + \langle   \mathbf{h}_{in},  \mathbf{g}(0,\cdot)\rangle_{\H^0} +\int_0^T \langle   \mathbf{q}(t,\cdot),  \mathbf{g}(t,\cdot)\rangle_{\H^0} \,dt \\
 =& \int_0^T \langle \mathbf{P}'(\mathbf{u}) \mathbf{h},  \mathbf{g}\rangle_{\H^0} \,dt - \langle   \mathbf{h}_{end},  \mathbf{g}_1\rangle_{\H^0} + \langle   \mathbf{h}_{in},  \mathbf{g}(0,\cdot)\rangle_{\H^0}  \\
 =& \int_0^T \langle  \mathbf{h},  [\mathbf{P}'(\mathbf{u})]^*\mathbf{g}\rangle_{\H^0} \,dt + \langle   \mathbf{h}|_{t=T},  \mathbf{g}_1\rangle_{\H^0} - \langle   \mathbf{h}|_{t=0},  \mathbf{g}_1\rangle_{\H^0} - \langle   \mathbf{h}_{end},  \mathbf{g}_1\rangle_{\H^0} + \langle   \mathbf{h}_{in},  \mathbf{g}(0,\cdot)\rangle_{\H^0} \\
 =& \langle \mathbf{h}|_{t=T}-  \mathbf{h}_{end},  \mathbf{g}_1\rangle_{\H^0}. 
\end{align*}
Since this relation holds for any $\mathbf{g}_1\in \H^0$, we conclude that $\mathbf{h}|_{t=T}=  \mathbf{h}_{end}$ \textit{a.e.}
This finishes the proof of the existence of such a control.

Next we prove the uniqueness of the control. 
Suppose there exists another $\tilde{\mathbf{f}} \in C([0,T]; \H^0(\T))$ that solves $[\mathbf{P}'(\mathbf{u})]^*\tilde{\mathbf{f}} = \mathbf{0}$, and the unique solution $\mathbf{h}$ of the Cauchy problem \eqref{H0ControlEqn} satisfies $\mathbf{h}|_{t=T} = \mathbf{h}_{end}$.
Let $\tilde{\mathbf{f}}|_{t=T} = \tilde{\mathbf{f}}_1$.
$\tilde{\mathbf{f}}_1$ satisfies \eqref{LaxMilgram}.
Hence, $\tilde{\mathbf{f}}_1$ is unique because the function $\mathbf{f}_1$ that satisfies \eqref{LaxMilgram} is unique.
We must have $\tilde{\mathbf{f}}_1 = \tilde{\mathbf{f}}$.
$\mathbf{f}$ is unique because it is the unique solution of the first backward adjoint system in \eqref{mLBackwardAdjoint}.

If $\mathbf{h}_{in}, \mathbf{h}_{end} \in \H^0(\T),$ and $\mathbf{q}\in C([0,T]; \H^0(\T))$ are real-valued functions, one  modifies $\langle   \cdot,  \cdot\rangle_{\H^0_x}$ by the real inner product
\begin{equation*}
\langle  (u_1, u_2),  (v_1, v_2)\rangle_{\H^0_x} : = \int_\T |D|^\f32 u_1(x) |D|^\f32v_1(x) + u_2(x)v_2(x) \,dx.
\end{equation*}
Then we can obtain the real-valued solution $\mathbf{f}$.
\end{proof}

We then consider the $\H^s$ controllability of the linearized system \eqref{PuTildeu}.
\begin{lemma}\label{t:ControlHs}
Let $T>0$ and $\omega \subset \T$ be an open set. 
There exists $\delta_s\in (0,1)$ small enough and a positive universal constant $s_1$ such that, if $\| (\eta, \psi)\|_{C([0,T]; \H^{s+\f52})\cap C^1([0,T]; \H^{s})\cap C^2([0,T]; \H^{s-\f52})}<\delta_s$ for $s\geq s_1$, then for any real-valued $\mathbf{h}_{in}, \mathbf{h}_{end} \in \H^s(\T)$, and $\mathbf{q}\in C([0,T]; \H^s(\T))$, the control function $\mathbf{f}$ constructed in Lemma \ref{t:ControlH0} is in $C([0,T]; \H^s(\T))$.
Furthermore, $\mathbf{f}$ and  $\mathbf{h}$ satisfy the estimate
\begin{equation*}
 \|(\mathbf{f},\mathbf{h})\|_{C([0,T]; \H^s)} \lesssim \|\mathbf{h}_{in}\|_{\H^s} + \|\mathbf{h}_{end}\|_{\H^s} + \|\mathbf{q}\|_{C([0,T]; \H^s)}.
\end{equation*}
If $\mathbf{h}_{in}$, $\mathbf{h}_{end}\in\mH^{s+5}$, $\mathbf{q}\in C([0,T];\mH^{s+5})\cap C^1([0,T];\mH^{s+\f52})$, then $\mathbf{h}\,,\mathbf{f}\in C([0,T]; \H^{s+5})\cap C^1([0,T]; \H^{s+\frac{5}{2}})\cap C^2([0,T]; \H^{s}) $, and
\begin{align*}
&\|(\mathbf{f},\mathbf{h})\|_{C([0,T]; \H^{s+5})} + \|\p_t(\mathbf{f},\mathbf{h})\|_{C([0,T]; \H^{s+\f52})} +\|\p_{tt}(\mathbf{f},\mathbf{h})\|_{C([0,T]; \H^{s})}\\
&\lesssim \|\mathbf{h}_{in}\|_{\H^{s+5}} + \|\mathbf{h}_{end}\|_{\H^{s+5}} + \|\mathbf{q}\|_{C([0,T]; \H^{s+5})}+\|\p_t\mathbf{q}\|_{C([0,T]; \H^{s+\f52})}.
\end{align*}
\end{lemma}

\begin{proof}
Let $\mathbf{f}, \mathbf{h}\in C([0,T]; \H^0)$ be the solution of the control problem constructed in Lemma \ref{t:ControlH0}.
Recall that from the computation in Section \ref{s:Top} and Section \ref{s:Lower}, the operator $\mathbf{P}'(\mathbf{u})$ can be reduced to $\mL_8 + \mR_8$ by
\begin{equation*}
 \mathbf{P}'(\mathbf{u}) =   \mZ\mB\mA PSM\mathcal{T}\mathcal{O}\mU(\mL_8 + \mR_8)\mU^{-1} \mathcal{O}^{-1}\mathcal{T}^{-1}M^{-1}S^{-1}Q^{-1}\mA^{-1}\mB^{-1}\mZ^{-1}.
\end{equation*}
We write
\begin{align*}
&\Phi := \mZ\mB\mA PSM\mathcal{T}\mathcal{O}\mU, \quad \Psi: = (\mZ\mB\mA QSM\mathcal{T}\mathcal{O}\mU)^{-1},\quad \widetilde{\mathbf{h}} := |D|^{\f32}\Psi\mathbf{h}, \\
&\widetilde{\mathbf{h}}_{in}: =|D|^{\f32}\Psi|_{t=0} \mathbf{h}_{in}, \quad \widetilde{\mathbf{h}}_{end}: =|D|^{\f32}\Psi|_{t=T} \mathbf{h}_{end}, \quad \widetilde{\mathbf{q}} := |D|^{\f32}\Phi^{-1}  \mathbf{q}, \\
&\widetilde{\mathbf{f}}: =|D|^{\f32}S^{-2}\Phi^{*}\mathbf{f}, \quad K =|D|^{\f32}  \Phi^{-1} \text{diag}\{0, \chi_\omega \}(\Phi^{*})^{-1}S^2|D|^{-\f32},
\end{align*}
where $\Phi_*$ is the adjoint operator of $\Phi$  with respect to the space-time scalar product $\langle \cdot, \cdot\rangle_{(t,x)} = \int_0^T \langle \cdot, \cdot\rangle_{L^2_x\times L^2_x}dt$. 
We compute
\begin{align*}
\mB^* = \mB^{-1}, \quad \mA^* = \mA^{-1}, \quad P^* = P, \quad S^* = S, 
\quad \mathcal{T}= \mathcal{T}^{-1}, \quad \mathcal{O}^{*} =  \mathcal{O}^{-1}, \quad \mZ^*= \mZ^T,
\end{align*}
\begin{align*}
\mathcal{U}^* =  \begin{pmatrix}
 (A^{(1)})^*   & 0 \\ 
0 &  (A^{(2)})^*
\end{pmatrix},\quad
    M^*= I +  \begin{pmatrix} 
0 &    0\\ 
|D_y|^{-5/2}(\ell_1\cdot) &  \mH |D_y|^{-1}(v_1\cdot)+   |D_y|^{-2}(v_2\cdot) \end{pmatrix},
\end{align*}
\begin{align*}
K=&|D|^{\f32}\mU^{-1}\mathcal{O}^{-1}\mathcal{T}^{-1}M^{-1}S^{-1}P^{-1}\mA^{-1}\mB^{-1}\mZ^{-1}\chi_{\omega}(\mU^{*}\mathcal{O}^{*}\mathcal{T}^{*}M^{*}S^{*}P^{*}\mA^{*}\mB^{*}\mZ^{*})^{-1}S^2|D|^{-\f32}\\
=&|D|^{\f32}\mU^{-1}\mathcal{O}^{-1}\mathcal{T}^{-1}M^{-1}S^{-1}P^{-1}\mA^{-1}\mB^{-1}\mZ^{-1}\chi_{\omega}(\tilde{\mZ}^*)^{-1}\mB\mA P^{-1}S^{-1}(M^*)^{-1}\mathcal{T}\mathcal{O}(\mathcal{U}^*)^{-1}S^2|D|^{-\f32}.
\end{align*}
Then $K$ is an operator of order $0$.
By construction, one can check that $\widetilde{\mathbf{h}}$ and $\widetilde{\mathbf{f}}$ satisfy
\begin{align*}
\begin{cases}
(\mL_8 + |D|^{\f32}\mR_8|D|^{-\f32})\widetilde{\mathbf{h}}=K\widetilde{\mathbf{f}}+\widetilde{\mathbf{q}}\\
\widetilde{\mathbf{h}}(0,\cdot)=\widetilde{\mathbf{h}}_{in}\\
\widetilde{\mathbf{h}}(T,\cdot)=\widetilde{\mathbf{h}}_{end},\\
\end{cases}
\quad
(\mL_8 + |D|^{\f32}\mR^*_8|D|^{\f32})\widetilde{\mathbf{f}}= \mathbf{0}.
\end{align*}

As in \cite{BHM}, we follow an argument used by Dehman-Lebeau  \cite{MR2486082}. Since we can remove the hypothesis that $\mathbf{h}_{end}$ and $\mathbf{q}$ are zero functions by introducing an auxiliary function $\mathbf{w}$, which is the solution of the backward Cauchy problem
\begin{align} \label{Puwq}
\mathbf{P}'(\mathbf{u}) \mathbf{w} = \mathbf{q}, \quad  \mathbf{w}(T,\cdot)=\mathbf{h}_{end},
\end{align}
we therefore assume $\mathbf{h}_{end}=\mathbf{q}= \mathbf{0}$ in the following.

Define the map 
\begin{align*}
S_1: L^2\times L^2 \rightarrow L^2\times L^2, \quad S_1\widetilde{\mathbf{f}}_1=\tilde{\mathbf{h}}(0,\cdot),
\end{align*}
where $\widetilde{\mathbf{f}}$  and  $\tilde{\mathbf{h}}$ are the solutions of the Cauchy problems
\begin{equation}\label{Col-Cauchy}
\begin{cases}
(\mL_8 + |D|^{\f32}\mR^*_8|D|^{-\f32})\widetilde{\mathbf{f}}=\mathbf{0}\\
\widetilde{\mathbf{f}}(T,\cdot)=\widetilde{\mathbf{f}}_1,\\
\end{cases}
\quad
\begin{cases}
(\mL_8 + |D|^{\f32}\mR_8|D|^{-\f32})\widetilde{\mathbf{h}}=K\widetilde{\mathbf{f}}\\
\widetilde{\mathbf{h}}(T,\cdot)= \mathbf{0}.
\end{cases}
\end{equation}
By existence and uniqueness results in Lemma \ref{t:ControlH0}, it follows that $S_1$ is a linear isomorphism. We also consider 
\begin{equation}\label{Col-Cauchy1}
\begin{cases}
(\mL_8 + |D|^{\f32}\mR^*_8|D|^{-\f32})\underline{\widetilde{\mathbf{f}}}= \mathbf{0}\\
\underline{\widetilde{\mathbf{f}}}(T,\cdot)=|D|^s\widetilde{\mathbf{f}}_1,\\
\end{cases}
\quad
\begin{cases}
(\mL_8 +|D|^{\f32}\mR_8|D|^{-\f32})\underline{\widetilde{\mathbf{h}}}=K\underline{\widetilde{\mathbf{f}}}\\
\widetilde{\mathbf{h}}(T,\cdot)= \mathbf{0}.
\end{cases}
\end{equation}
The difference $|D|^s\widetilde{\mathbf{f}}-\underline{\widetilde{\mathbf{f}}}$ satisfies
\begin{align*}
\begin{cases}
(\mL_8 +|D|^{\f32}\mR^*_8|D|^{-\f32})(|D|^s\widetilde{\mathbf{f}}-\underline{\widetilde{\mathbf{f}}})=[|D|^{\f32}\mR^*_8|D|^{-\f32},|D|^s]\widetilde{\mathbf{f}}\\
(|D|^s\widetilde{\mathbf{f}}-\underline{\widetilde{\mathbf{f}}})(T,\cdot)=\mathbf{0}.
\end{cases}
\end{align*}
By Lemma \ref{t:CauchyEight}, and using the fact that $\mR_8^*$ is an operator of order $0$,
\begin{align*}
 \| |D|^s\widetilde{\mathbf{f}}-\underline{\widetilde{\mathbf{f}}}\|_{C([0,T]; L^2\times L^2)}\lesssim_s e^{CT}\|[|D|^{\f32}\mR^*_8|D|^{-\f32},|D|^s]\widetilde{\mathbf{f}}\|_{C([0,T]; L^2\times L^2)} \lesssim_s e^{CT}\|\widetilde{\mathbf{f}} \|_{C([0,T]; H^{s-1}\times H^{s-1})}.  
\end{align*}
The difference $|D|^s\widetilde{\mathbf{h}}-\underline{\widetilde{\mathbf{h}}}$ satisfies
\begin{align*}
\begin{cases}
(\mL_8 + |D|^{\f32}\mR^*_8|D|^{-\f32})(|D|^s\widetilde{\mathbf{h}}-\underline{\widetilde{\mathbf{h}}})=K(|D|^s\widetilde{\mathbf{f}}-\underline{\widetilde{\mathbf{f}}})+[|D|^{\f32}\mR^*_8|D|^{-\f32},|D|^s]\widetilde{\mathbf{h}}+[|D|^s,K]\widetilde{\mathbf{f}}\\
(|D|^s\widetilde{\mathbf{h}}-\underline{\widetilde{\mathbf{h}}})(T,\cdot)= \mathbf{0}.
\end{cases}
\end{align*}
Again using Lemma \ref{t:CauchyEight}, and the fact that the operator $K$ is of order $0$, we have
\begin{align*}
 \| |D|^s\widetilde{\mathbf{h}}-\underline{\widetilde{\mathbf{h}}}\|_{C([0,T]; L^2\times L^2)} 
 \lesssim_s& e^{CT}\Big(\|K(|D|^s\widetilde{\mathbf{f}}-\underline{\widetilde{\mathbf{f}}})\|_{C([0,T];L^2\times L^2)} \\
 +& \|[|D|^{\f32}\mR^*_8|D|^{-\f32},|D|^s]\widetilde{\mathbf{h}}\|_{C([0,T];L^2\times L^2)} +\|[|D|^s,K]\widetilde{\mathbf{f}}\|_{C([0,T];L^2\times L^2)}\Big)\\
 \lesssim_s& e^{CT}\left(\|\widetilde{\mathbf{h}} \|_{H^{s-1}\times H^{s-1}} + \|\widetilde{\mathbf{f}}\|_{C([0,T]; H^{s-1}\times H^{s-1})}\right).  
\end{align*}
Applying Lemma \ref{t:CauchyEight} for \eqref{Col-Cauchy}, we have
\begin{equation}\label{Col-Cauchy-es}
\begin{aligned}
\|\widetilde{\mathbf{f}}\|_{C([0,T]; H^s\times H^s)} \lesssim_s& e^{CT} \|\widetilde{\mathbf{f}}_1\|_{C([0,T]; H^{s}\times H^{s})},\\
\|\widetilde{\mathbf{h}}\|_{C([0,T]; H^s\times H^s)} \lesssim_s&e^{CT} \|\widetilde{\mathbf{f}}\|_{C([0,T]; H^{s}\times H^{s})}\lesssim_s e^{CT} \|\widetilde{\mathbf{f}}_1\|_{C([0,T]; H^{s}\times H^{s})}.
\end{aligned}
\end{equation}
Thus, we have
\begin{align*}
  \| |D|^s\widetilde{\mathbf{f}}-\underline{\widetilde{\mathbf{f}}}\|_{C([0,T]; L^2\times L^2)}, \| |D|^s\widetilde{\mathbf{h}}-\underline{\widetilde{\mathbf{h}}}\|_{C([0,T]; L^2\times L^2)} \lesssim_s e^{CT} \|\widetilde{\mathbf{f}}_1\|_{C([0,T]; H^{s-1}\times H^{s-1})}.
\end{align*}
From the definition of the map $S_1$, we have $\underline{\mathbf{h}}(0,\cdot)=S_1|D|^s\widetilde{\mathbf{f}}_1$ and $S_1\widetilde{\mathbf{f}}_1=\widetilde{\mathbf{h}}_{in}$, then by the above estimate
\begin{align*}
  \|S_1|D|^s\widetilde{\mathbf{f}}_1\|_{L^2\times L^2} \lesssim& \||D|^s\widetilde{\mathbf{h}}(0,\cdot)\|_{L^2\times L^2} +\||D|^s\widetilde{\mathbf{h}}(0,\cdot)-\underline{\mathbf{h}}(0,\cdot)\|_{L^2\times L^2} \\
  \lesssim&
\|\widetilde{\mathbf{h}}_{in}\|_{H^s\times H^s}+\||D|^s\widetilde{\mathbf{h}}-\underline{\mathbf{h}}\|_{C([0,T];L^2\times L^2)}\\
\lesssim_s&\|\widetilde{\mathbf{h}}_{in}\|_{H^s\times H^s}+e^{CT} \|\widetilde{\mathbf{f}}_1\|_{C([0,T]; H^{s-1}\times H^{s-1})}.
\end{align*}
Since $S_1$ is a linear isomorphism, we then have
\begin{align*}
\|\widetilde{\mathbf{f}}_1\|_{C([0,T]; H^{s}\times H^{s})}\lesssim_s \|\widetilde{\mathbf{h}}_{in}\|_{H^s\times H^s}+e^{CT} \|\widetilde{\mathbf{f}}_1\|_{C([0,T]; H^{s-1}\times H^{s-1})}. 
\end{align*}
Note that $\|\widetilde{\mathbf{f}}_1\|_{L^2\times L^2}\lesssim \|\widetilde{\mathbf{h}}_{in}\|_{L^2\times L^2}$, by induction on $s$, we get that
\begin{align*}
\|\widetilde{\mathbf{f}}_1\|_{C([0,T]; H^{s}\times H^{s})}\lesssim_s \|\widetilde{\mathbf{h}}_{in}\|_{H^s\times H^s}.
\end{align*}
By \eqref{Col-Cauchy-es}, we then get
\begin{align*}
\|\widetilde{\mathbf{f}}\|_{C([0,T]; H^s\times H^s)}+\|\widetilde{\mathbf{h}}\|_{C([0,T]; H^s\times H^s)}\lesssim_s \|\widetilde{\mathbf{h}}_{in}\|_{H^s\times H^s}.
\end{align*}
Finally, using the estimates in Section \ref{s:Well}, we obtain that
\begin{align*}
\|\mathbf{f}\|_{C([0,T]; \H^s)}+\|\mathbf{h}\|_{C([0,T]; \H^s)}\lesssim_s \|\mathbf{h}_{in}\|_{\H^s}.
\end{align*}
Adding the auxiliary function $\mathbf{w}$ defined in \eqref{Puwq}, we obtain that
\begin{equation*}
 \|\mathbf{f}\|_{C([0,T]; \H^s)} + \|\mathbf{h}\|_{C([0,T]; \H^s)} \lesssim \|\mathbf{h}_{in}\|_{\H^s} + \|\mathbf{h}_{end}\|_{\H^s} + \|\mathbf{q}\|_{C([0,T]; \H^s)}.
\end{equation*}
The estimates in Lemma \ref{t:ControlHs} are deduced from the fact that $\mathbf{f},\mathbf{h}$ solve the systems $\mathbf{P}'(\mathbf{u}) \mathbf{h} =\text{diag}\{0, \chi_\omega \}\mathbf{f}+ \mathbf{q}$, and $[\mathbf{P}'(\mathbf{u})]^*\mathbf{f}= \mathbf{0}$.
\end{proof}

\section{Controllability of the full hydroelastic waves} \label{s:Full}
In this section, we finish the proof of Theorem \ref{t:MainResult} by combining the results in previous sections and applying the Nash-Moser-H\"ormander Theorem \ref{t:Nash}.

To use the Nash-Moser-H\"ormander theorem, we first define the function spaces $E_s$ and $F_s$.
Let 
\begin{align*}
X_s : = C([0,T]; \H^{s+5})\cap C^1([0,T]; \H^{s+\frac{5}{2}})\cap C^2([0,T]; \H^{s}).
\end{align*}
We define $E_s : = X_s \times X_s$, and
\begin{align*}
F_s : = &\{\mathbf{z}: = (\mathbf{v}, \mathfrak{a}, \mathfrak{b}) = (v_1, v_2, \mathfrak{a}_1, \mathfrak{a}_2, \mathfrak{b}_1, \mathfrak{b}_2): \\
&\mathbf{v}\in C([0,T]; \H^{s+5})\cap C^1([0,T]; \H^{s+\frac{5}{2}}),\quad \mathfrak{a}, \mathfrak{b}\in \H^{s+5} \}.
\end{align*}
By employing the tame estimates from Lemma 7.1 and Lemma 7.2 of \cite{BHM}, we can replace the requirement that
$\| (\eta, \psi)\|_{C([0,T]; \H^{s+\f52})\cap C^1([0,T]; \H^{s})\cap C^2([0,T]; \H^{s-\f52})}$ be small in Lemma \ref{t:ControlHs} with the condition that $\|\mathbf{u}\|_{X_\mathfrak{c}}$ be small for some universal constant $\mathfrak{c}$. Consequently, we obtain the following result.
\begin{proposition}\label{t:Rightinv}
Let $T>0$ and $\omega \subset \T$ be an open set. 
There exists $\delta_*\in (0,1)$ small enough and a positive universal constant $\mathfrak{c}$ such that, if $\|\mathbf{u}\|_{X_\mathfrak{c}}< \delta_*$, then there exists $(\mathbf{h},\boldsymbol{\varphi})\in E_s$ such that
\begin{equation*}
\mathbf{P}'(\mathbf{u}) \mathbf{h} -\text{diag}\{0,\chi_\omega\}\boldsymbol{\varphi}= \mathbf{v}, \quad 
\mathbf{h}|_{t=0} =\mathfrak{a},\quad \mathbf{h}|_{t=T} =\mathfrak{b},
\end{equation*}
and
\begin{align*}
\|(\mathbf{h},\boldsymbol{\varphi})\|_{E_s} \lesssim_s \|(\mathbf{v},\mathfrak{a},\mathfrak{b})\|_{F_s}+\|\mathbf{u}\|_{X_{s+\mathfrak{c}}}\|(\mathbf{v},\mathfrak{a},\mathfrak{b})\|_{F_0}).
\end{align*}    
\end{proposition}

We define the smoothing operators $S_j$, $j=0,1,2,...$ as
\begin{align*}
S_ju(x):=\sum_{|k|\leq 2^j}\hat{u}_ke^{ikx},\quad \text{where} \quad u(x)=\sum_{k\in\mathbb{Z}}\hat{u}_ke^{ikx}.
\end{align*}
One easily
verifies that $S_j$ satisfies the assumptions of Theorem \ref{t:Nash} on $E_s$ and $F_s$.

From \eqref{e:hydro}, $\Phi(\mathbf{u},f)$  belongs to $F_s$ when $(\mathbf{u},0,f)\in E_{s+\f52}$. 
From \eqref{t:Lannes} and \eqref{SecondElastic}, we have if $\|\tilde{\mathbf{u}}_2\|_{X_2}\leq\delta_0$, with $\delta_0$ small enough, then
\begin{align*}
\|P''(\mathbf{u})[\tilde{\mathbf{u}}_1, \tilde{\mathbf{u}}_2]\|_{F_s}\lesssim_s\|\tilde{\mathbf{u}}_1\|_{X_{s+4}}\|\tilde{\mathbf{u}}_2\|_{X_{2}}+\|\tilde{\mathbf{u}}_2\|_{X_{s+4}}\|\tilde{\mathbf{u}}_1\|_{X_{2}}+\|\mathbf{u}\|_{X_{s+4}}\|\tilde{\mathbf{u}}_1\|_{X_2}\|\tilde{\mathbf{u}}_2\|_{X_2}.
\end{align*}
We fix $V = \{(\mathbf{u},0, f)\in E_4: \|(\mathbf{u}, 0,f)\|_{E_4}\leq\delta_0\} $ to be a small convex neighborhood of $E_4$, and choose $\delta_1=\delta_*$,
$a_0=2$, $\mu=4$, $a_1=\mathfrak{c}$, $\al=\beta>2\mathfrak{c}$ in Theorem \ref{t:Nash}.  
The right inverse in Proposition \ref{t:Rightinv} satisfies the assumptions of Theorem \ref{t:Nash}. Let $\mathbf{u}_{in}, \mathbf{u}_{end}\in \mH^{\beta+5}$, with $\|\mathbf{u}_{in}\|_{\mH^{\beta+5}}+ \| \mathbf{u}_{end}\|_{\mH^{\beta+5}}$ small enough. 
Let $\mathbf{g} := (\textbf{0}, \mathbf{u}_{in}, \mathbf{u}_{end})$, so that $\mathbf{g}\in F_{\beta}$ and $\|\mathbf{g}\|_{F_{\beta}}\leq\delta$. Since $\mathbf{g}$ does not depend on time, it satisfies \eqref{NineTwelve}.

By Theorem \ref{t:Nash}, there exists a solution $(\mathbf{u},0,f)$ of the system 
\begin{equation*}
\Phi(\mathbf{u},f) = \mathbf{g}, \quad \text{with } \|(\mathbf{u},0,f)\|_{E_{\al}}\leq \|\mathbf{g}\|_{F_{\al}}.
\end{equation*}
The function $\chi_\omega f$ is just the exterior pressure $P_{ext}(t,x)$ that we need in Theorem \ref{t:MainResult}.
It suffices to prove the uniqueness of the solution $\bu$.
Let $\bu_1$, $\bu_2$ be two solutions of \eqref{e:hydro}. We compute
\begin{align*}
\mathbf{P}(\bu_1)-\mathbf{P}(\bu_2)=\int_0^1\mathbf{P}'(\bu_2+\la(\bu_1-\bu_2))d\la[\bu_1-\bu_2]=:\tilde{\mathbf{P}}[\bu_1-\bu_2].
\end{align*}
The difference $\bu_1-\bu_2$ satisfies $\tilde{\mathbf{P}}[\bu_1-\bu_2]= \mathbf{0}$, $(\bu_1-\bu_2)(0,x)= \mathbf{0}$. 
It follows from the uniqueness of the Cauchy problem in Lemma \ref{t:CauchyZero} that $\bu_1=\bu_2$.

\appendix
\section{Function spaces and  estimates} \label{s:Appendix}
\subsection{Definition of functions and some estimates}
Here, we recall the definition of function spaces and some of the estimates that are used in previous sections.

For a periodic function $f(x)\in L^2(\mathbb{T}; \mathbb{C})$, it can be expanded in Fourier series as
\begin{equation*}
f(x) = \frac{1}{\sqrt{2\pi}}\sum_{j\in \mathbb{Z}} \hat{f}(j)e^{ijx}, \quad \hat{f}(j): = \frac{1}{\sqrt{2\pi}}\int_{0}^{2\pi}f(x)e^{-ijx}dx.
\end{equation*}
Most of our analysis in this paper happens at the level of  homogeneous Sobolev spaces $\dot{H}^s(\mathbb{T})$. 
Their norms are given by
\begin{equation*}
    \|f\|_{\dot{H}^s(\mathbb{T})} = \left( \sum_{n\in \mathbb{Z}} |n|^{2s} |\hat{f}(n)|^2\right)^{\frac{1}{2}}. 
\end{equation*}
And we define $\dot{H}^0(\mathbb{T}) =L^2(\mathbb{T})$.
For functions without zero frequency mode on the torus, the homogeneous and the usual inhomogeneous Sobolev spaces are equivalent:
\begin{equation*}
 \|f\|_{\dot{H}^s(\mathbb{T})} \leq \|f\|_{H^s(\mathbb{T})}, \quad \|f\|_{H^s(\mathbb{T})} \leq 2^\frac{s}{2}\|f\|_{\dot{H}^s(\mathbb{T})}, \quad s\geq 0.
\end{equation*}
Let $q: \mathbb{R} \rightarrow \mathbb{T}$ be the quotient map.
We define the Zygmund space of $f$ on the torus as its extension on the real line $\mathbb{R}$.
More precisely, we set
\begin{equation*}
    \|f\|_{C^s_*(\mathbb{T})} = \|q^{*}f\|_{C^s_*(\mathbb{R})},
\end{equation*}
where $q^{*}f(x) = f(q(x))$ is the pullback of $f$ over $q$.

For the estimates in $\dot{H}^s$, we have the product and Moser type estimates.
\begin{lemma}
    Let $s>0$, then $\dot{H}^s \cap L^\infty$ is an algebra, with the estimate
\begin{equation}
\|uv\|_{\dot{H}^s} \lesssim \|u\|_{\dot{H}^s} \|v\|_{L^\infty} + \|u\|_{L^\infty}\|v\|_{\dot{H}^s}.
\end{equation}
When $s>\frac{1}{2}$, by Sobolev embedding,
\begin{equation}
  \|uv\|_{\dot{H}^s} \lesssim \|u\|_{\dot{H}^s} \|v\|_{\dot{H}^s}.  
\end{equation}
 Furthermore, for any smooth function $F$ that vanishes at $0$, then the following Moser type  estimates hold:
 \begin{equation}
\|F(u)\|_{\dot{H}^s}\lesssim_{\|u\|_{L^\infty}} \|u\|_{\dot{H}^s}. \label{MoserTwo}
    \end{equation}
\end{lemma}

We recall the result from Appendix $B$ of \cite{BHM}  for the Sobolev estimates of composition with a spatial diffeomorphism.
\begin{lemma}[\hspace{1sp}\cite{BHM}] \label{t:BaldiLemma}
There exists $\delta \in (0, 1)$ with the following properties.

(1) \textit{Let $s \geq 0$ and $\alpha \in C([0, T], H^{s+2}(\mathbb{T}))$, with $\|\alpha\|_{C([0,T];H^2)} \leq \delta$. Then the operator $\mathcal{B}u(t, x) := u(t, x + \alpha(t, x))$ is a linear and continuous operator $C([0, T], H^s(\mathbb{T})) \to C([0, T], H^s(\mathbb{T}))$, with}
\begin{equation} \label{BtransformBound}
    \|\mathcal{B}u\|_{C([0,T];H^s)} \lesssim_s \|u\|_{C([0,T];H^s)} + \|\alpha\|_{C([0,T];H^{s+2})}\|u\|_{C([0,T];L^2)}. 
\end{equation}

(2) \textit{Let $s \geq 0$ and $\alpha \in C([0, T], H^{s+4}(\mathbb{T}))$, with $\|\alpha\|_{C([0,T];H^4)} \leq \delta$. Then the inverse operator $\mathcal{B}^{-1}$, defined by $\mathcal{B}^{-1}u(t, y) := u(t, y + \widetilde{\alpha}(t, y))$, maps $C([0, T], H^s(\mathbb{T}))$ into itself, with}
\[
    \|\mathcal{B}^{-1}u\|_{C([0,T];H^s)} \lesssim_s \|u\|_{C([0,T];H^s)} + \|\alpha\|_{C([0,T];H^{s+4})}\|u\|_{C([0,T];L^2)}, \qquad \forall u \in C([0, T], H^s(\mathbb{T})) .
\]
\end{lemma}

\begin{lemma}[\hspace{1sp}\cite{MR3011291}] \label{t:BaldiCommutator}
(1) Let $m_1, m_2 \in \mathbb{N}$, with $s \geqslant 2$, $m_1, m_2 \geqslant 0$, $m = m_1 + m_2$. Let $f(t, \cdot) \in H^{s+m}(\mathbb{T}, \mathbb{C})$. Then $[f, \mathcal{H}]u = f\mathcal{H}u - \mathcal{H}(fu)$ satisfies
\begin{equation} \label{HilbertCommutator}
 \left\| \partial_x^{m_1} [f, \mathcal{H}] \partial_x^{m_2} u \right\|_{H^s} \leqslant C(s) \bigl( \|f\|_{H^{m+2}}\|u\|_{H^s}  + \|f\|_{H^{m+s+2}}\|u\|_{L^2}  \bigr).   
\end{equation}

\noindent (2) Let $a : \mathbb{T} \to \mathbb{T}$ be a function, and $Au(t, x) = u(a(t), x)$. Then $[A, \mathcal{H}] = 0$.

\noindent (3) There exists a universal constant $\delta \in (0, 1)$ with the following property. Let $m_1, m_2 \in \mathbb{N}$, $s\geq 1$, $m = m_1 + m_2$, $\beta(t, \cdot) \in W^{s+m+1, \infty}(\mathbb{T}, \mathbb{R})$, with $\|\beta\|_{W^{1,\infty}} \leqslant \delta$. Let $\mB h(t, x) = h(t, x + \beta(t, x))$, $h \in H^s(\R \times\mathbb{T}, \mathbb{C})$. Then
\begin{equation} \label{BConjugationDiff}
\left\| \partial_x^{m_1} \bigl( \mB^{-1}\mathcal{H} \mB - \mathcal{H} \bigr) \partial_x^{m_2} h \right\|_{\dot{H}^s} \leqslant C(s, m) \bigl( \|\beta\|_{W^{m+1}} \|h\|_{\dot{H}^s} + \|\beta\|_{\dot{H}^{s+m+1}} \|h\|_{L^2} \bigr).
\end{equation}
\end{lemma}

\subsection{Estimates for the pseudo-differential operators}

Given $\al \in \R$, the operator $|D_x|^\al$ is the Fourier multiplier with symbol $g_\al(\xi)$, where $g_\al:\R \rightarrow \R$ is a $C^\infty$ function such that $g_\al(\xi) \geq 0$ for all $\xi \in \R$, $g_\al(\xi) = |\xi|^\al$ for $|\xi| \geq \f23$, and $g_\al(\xi) = 0$ for $|\xi| \leq \f13$.
For the periodic function $h(x) = \sum_{j\in \Z}h_j e^{ijx}$, one has $|D_x|^\al h(x) = \sum_{j\neq 0}h_j |j|^\al e^{ijx}$, for any $\al \in \R$.
Other differential operators on the periodic functions are defined similarly.

For a Fourier multiplier $\Lambda$ and any multiplication operator $h \mapsto a h$, we have the following symbolic composition formula.
\begin{equation}  \label{SymbolExpansion}
\Lambda (a u) \sim \sum_{n=0}^\infty \frac{1}{i^n n!} \, (\p_x^n a)(x)   \mathrm{Op}(\p^n_\xi g) u,
\end{equation}
where Op$(\p^n_\xi g)$ is the Fourier multiplier with symbol $\p^n_\xi g(\xi)$. 

Consider the periodic Fourier integral operator of amplitude $a$ and phase function $\phi$, 
\begin{equation} \label{def A FIO with amplitude}
Au(x) = \sum_{k \in \Z} \hat u_k \, a(x,k) \, e^{i \phi(x,k)}, \qquad 
\phi(x,k) = kx + |k|^{1/2} \b(x).
\end{equation}
We have the following result on the composition formula.

\begin{lemma}[\hspace{1sp}\cite{MR3356988}] \label{t:Composition} 
Let $A$ be the operator \eqref{def A FIO with amplitude}, 
with $\|\b\|_{W^{1,\infty}} \leq 1/4$ and $\| \b \|_{H^2} \leq 1/2$. 
Let 
\[
r,m,s_0 \in \R, \quad 
m \geq 0, \quad
s_0 > 1/2, \quad 
N \in \N, \quad 
N \geq 2(m + r + 1) + s_0. 
\] 
Then, we have the composition formula.
\begin{align*}
&|D_x|^r A u = \sum_{\a=0}^{N-1} B_\a u + R_N u, \\   
B_\a u(x) 
= \binom{r}{\a} \,  \sum_{k \in \Z\backslash \{ 0\}} &|k|^{r-\a} \, (- i \, \text{sgn} \, k)^\a \, \hat u_k \,  e^{ikx} \, \p_x^\a \big\{ a(x,k) e^{i f(k) \b(x)} \big\},
\end{align*}
namely 
\[
B_\a u = F_\a |D_x|^{r-\a} \mH^\a u,
\quad 
F_\a v(x) := 
\binom{r}{\a} \, 
\sum_{k \in \Z} \hat v_k \, e^{ikx} \, \p_x^\a \big\{ a(x,k) e^{i f(k) \b(x)} \big\},
\]
and 
$\binom{r}{\a} \, := \frac{r (r-1) \ldots (r-\a+1)}{\a!}$. 
For every $s \geq s_0$, the remainder satisfies
\begin{equation} \label{final tame estimate}
\| R_N |D_x|^m u \|_{H^s} 
\leq 
C(s) \Big\{ \mK_{2(m+r+s_0+1)} \, \| u \|_{H^s}
+ \mK_{s + N + m + 2} \, \| u \|_{H^{s_0}} \Big\},
\end{equation}
where $\mK_\mu := | a |_{\mu} + | a |_{1}  \| \b \|_{H^{\mu + 1}}$ for $\mu \geq 0$ 
and $| a |_\mu := \sup_{k \in \Z} \| a(\cdot,k) \|_{H^\mu}$. 
Moreover,
\[
\mH |D_x|^r A u = \sum_{\a=0}^{N-1} B_\a \mH u + \tilde R_N u,
\]
where $\tilde R_N$ satisfies the same estimate \eqref{final tame estimate} as $R_N$.
\end{lemma}

We assume in addition that the amplitude $a$ is of order zero in $k$ and is a perturbation of $1$: $a(x,k) = 1 + b(x,k)\,$.
Denote $|b|_s := \sup_{k \in \Z} \| b(\cdot\,,k) \|_{H^s}$.
Then we have the following bound for the Fourier integral operator $A$.
\begin{lemma}[\hspace{1sp}\cite{MR3356988}] \label{t:Einv} 
There exist a universal constant $\delta > 0$ with the following properties. 
Let $\b \in H^3(\T)$ and $b(\cdot\,,k) \in H^3(\T)$ for all $k \in \Z$. 
If $\| \b \|_{H^3} + |b|_3 \leq \delta$, then $A$ is invertible from $L^2(\T)$ onto itself, with 
\[
\| A u \|_{L^2} + \| A\inv u \|_{L^2} + \|A^* u\|_{L^2} +  \|(A^*)^{-1} u\|_{L^2} \leq C \, \| u \|_{L^2} \,,
\]
where $C>0$ is a universal constant, and $A^*$ is the adjoint operator of $A$. 
If in addition, $\a \geq 1$ is an integer, $\b \in H^{\a+2}(\T)$ and $b(\cdot\,,k) \in H^{\a+2}(\T)$ for all $k \in \Z$, then 
\begin{equation*}
\| A u \|_{H^\a} + \| A\inv u \|_{H^\a} + \| A^* u \|_{H^\a} + \| (A^*)\inv u \|_{H^\a}
\leq C(\a) \, \Big( \| u \|_{H^\a} + (|b|_{\a+2} + \| \b\|_{H^{\a+2}}) \| u \|_{H^1} \Big), 
\end{equation*}
where $C(\a)>0$ depends only on $\a$.
\end{lemma}

\subsection{Nash-Moser-H\"ormander theorem}
In this section, we recall the Nash-Moser-H\"ormander theorem proved in \cite{MR3711883}.
Let $(E_a)_{a\geq 0}$ be a decreasing family of Banach spaces with continuous injections $E_b \hookrightarrow E_a$,
\begin{equation}
    \|u\|_a \leq \|u\|_b \quad \text{for } a \leq b. 
\end{equation}
Set $E_\infty = \cap_{a \geq 0} E_a$ with the weakest topology making the injections $E_\infty \hookrightarrow E_a$ continuous. 
Assume that $S_j : E_0 \to E_\infty$ for $j = 0, 1, \dots$ are linear operators such that, with constants $C$ bounded when $a$ and $b$ are bounded, and independent of $j$,
\begin{align}
    \|S_j u\|_a &\leq C \|u\|_a && \text{for all } a;  \\
    \|S_j u\|_b &\leq C 2^{j(b-a)} \|S_j u\|_a && \text{if } a < b;  \\
    \|u - S_j u\|_b &\leq C 2^{-j(a-b)} \|u - S_j u\|_a && \text{if } a > b;  \\
    \|(S_{j+1} - S_j)u\|_b &\leq C 2^{j(b-a)} \|(S_{j+1} - S_j)u\|_a && \text{for all } a, b. 
\end{align}
Set
\begin{equation}
    R_0 u := S_1 u, \qquad R_j u := (S_{j+1} - S_j)u, \quad j \geq 1. 
\end{equation}
Thus,
\begin{equation}
    \|R_j u\|_b \leq C 2^{j(b-a)} \|R_j u\|_a \quad \text{for all } a, b. 
\end{equation}

We also assume that
\begin{equation}
    \|u\|_a^2 \leq C \sum_{j=0}^\infty \|R_j u\|_a^2 \quad \forall a \geq 0, 
\end{equation}
with $C$ bounded for $a$ bounded.

Now let us suppose that we have another family $F_a$ of decreasing Banach spaces with smoothing operators having the same properties as above. We use the same notation also for the smoothing operators.

\begin{theorem}[\hspace{1sp}\cite{MR3711883}] \label{t:Nash}
 Let $a_1, a_2, \alpha, \beta, a_0, \mu$ be real numbers with
\begin{equation*}
    0 \leq a_0 \leq \mu \leq a_1, \qquad a_1 + \frac{\beta}{2} < \alpha < a_1 + \beta, \qquad 2\alpha < a_1 + a_2. 
\end{equation*}
Let $V$ be a convex neighborhood of $0$ in $E_\mu$. Let $\Phi$ be a map from $V$ to $F_0$ such that $\Phi : V \cap E_{a+\mu} \to F_a$ is of class $C^2$ for all $a \in [0, a_2 - \mu]$, with
\begin{equation} \label{NineTen}
    \|\Phi''(u)[v, w]\|_a \leq C \big(\|v\|_{a+\mu}\|w\|_{a_0} + \|v\|_{a_0}\|w\|_{a+\mu} + \|u\|_{a+\mu}\|v\|_{a_0}\|w\|_{a_0}\big) 
\end{equation}
for all $u \in V \cap E_{a+\mu}$, $v, w \in E_{a+\mu}$. Also assume that $\Phi'(v)$, for $v \in E_\infty \cap V$ belonging to some ball $\|v\|_{a_1} \leq \delta_1$, has a right inverse $\Psi(v)$ mapping $F_\infty$ to $E_{a_2}$, and that
\begin{equation} \label{NineEleven}
    \|\Psi(v)g\|_a \leq C(\|g\|_{a+\beta-\alpha} + \|g\|_0 \|v\|_{a+\beta}) \quad \forall a \in [a_1, a_2]. 
\end{equation}
For all $A > 0$ there exist $\delta, C_1 > 0$ such that, for every $g \in F_\beta$ satisfying
\begin{equation} \label{NineTwelve}
    \sum_{j=0}^\infty \|R_j g\|_\beta^2 \leq A \|g\|_\beta^2, \quad \|g\|_\beta \leq \delta, 
\end{equation}
there exists $u \in E_\alpha$, with $\|u\|_\alpha \leq C_1 \|g\|_\beta$, solving $\Phi(u) = \Phi(0) + g$.

Moreover, let $c > 0$ and assume that \eqref{NineTen} holds for all $a \in [0, a_2 + c - \mu]$, $\Psi(v)$ maps $F_\infty$ to $E_{a_2+c}$, and \eqref{NineEleven} holds for all $a \in [a_1, a_2 + c]$. If $g$ satisfies \eqref{NineTwelve} and, in addition, $g \in F_{\beta+c}$ with
\begin{equation}
    \sum_{j=0}^\infty \|R_j g\|_{\beta+c}^2 \leq A_c \|g\|_{\beta+c}^2 
\end{equation}
for some $A_c$, then the solution $u$ belongs to $E_{\alpha+c}$, with $\|u\|_{\alpha+c} \leq C_{1,c} \|g\|_{\beta+c}$.
\end{theorem}
\textbf{Acknowledgments.}
Jiaqi Yang is supported by National Natural Science Foundation of China under Grant: 12471225.
\bibliography{Hydro}
\bibliographystyle{plain}
	
\end{document}